\documentclass[11pt]{article}
\usepackage{amssymb,amsmath,setspace,graphicx,multicol,wrapfig,amsfonts,amsthm,titling,url,array,parskip,enumitem,bbm,color,tikz}
\usepackage{hyperref}
\usepackage{placeins}
\usepackage{aurical,pbsi}
\usepackage[T1]{fontenc}
\usepackage[hmargin=2.5cm,vmargin=2.5cm]{geometry}

\numberwithin{equation}{section}

\newtheoremstyle{plain}
{3pt} 
{3pt} 
{} 
{} 
{\bfseries} 
{.} 
{.5em} 
{} 
\newtheorem{thm}{Theorem}[section]
\newtheorem{lemma}[thm]{Lemma}
\newtheorem{prop}[thm]{Proposition}
\newtheorem{cor}[thm]{Corollary}

\newtheoremstyle{definition}
{3pt} 
{3pt} 
{} 
{} 
{\bfseries} 
{.} 
{.5em} 
{} 
\newtheorem{defn}{Definition}[section]

\newtheorem{example}{Example}[section]

\newtheoremstyle{plain}
{3pt} 
{3pt} 
{} 
{} 
{\bfseries} 
{.} 
{.5em} 
{} 
\newtheorem*{remark}{Remark}

\DeclareMathOperator{\wt}{wt}

\DeclareMathOperator{\weight}{weight}
\DeclareMathOperator{\Prob}{Prob}

\DeclareMathOperator{\pr}{pr}
\DeclareMathOperator{\DAE}{dae}
\DeclareMathOperator{\MLQ}{MLQ}
\DeclareMathOperator{\WMLQ}{WMLQ}
\DeclareMathOperator{\RAT}{RAT}
\DeclareMathOperator{\AMLQ}{AMLQ}
\DeclareMathOperator{\tr}{tr}

\DeclareMathOperator{\ASEP}{2-ASEP}
\DeclareMathOperator{\TASEP}{TASEP}

\DeclareMathOperator{\TRAT}{TRAT}
\DeclareMathOperator{\mlq}{mlq}
\DeclareMathOperator{\trat}{trat}
\DeclareMathOperator{\rat}{rat}
\DeclareMathOperator{\lfree}{free_L}
\DeclareMathOperator{\ufree}{free_U}
\DeclareMathOperator{\urest}{urest}
\DeclareMathOperator{\mv}{mv}
\DeclareMathOperator{\Left}{Left}
\DeclareMathOperator{\Up}{Up}

\title{Toric tableaux and the inhomogeneous two-species TASEP on a ring}
\author{Olya Mandelshtam}

\begin{document}

\maketitle
\tableofcontents

\begin{abstract}
The inhomogeneous two-species TASEP on a ring is an exclusion process that describes particles of different species hopping clockwise on a ring with inhomogeneous rates given by parameters. We introduce a new object that we call \emph{toric rhombic alternative tableaux}, which are certain fillings of tableaux on a triangular lattice tiled with rhombi, and are in bijection with the well-studied \emph{multiline queues} of Ferrari and Martin. Using the tableaux, we obtain a formula for the stationary probabilities of the inhomogeneous two-species TASEP, which specializes to results of Ayyer and Linusson. We obtain, in addition, an explicit determinantal formula for these probabilities, and define a Markov chain on the tableaux that projects to the two-species TASEP on a ring. 
\end{abstract}

\section{Introduction}\label{sec_intro}

It is well known that many exclusion processes have remarkable combinatorial structure. For example, the asymmetric simple exclusion process with open boundaries has been studied extensively as a projection of a Markov chain on certain tableaux which have strong connections to a number of important combinatorial objects \cite{CW07, CW11, MV15, CMW17}. On the other hand, the multispecies exclusion process on a ring (i.e.~with periodic boundary conditions) has been found to have a beautiful connection to multiline queues, a construction of Ferrari and Martin \cite{FM07}. In this paper we unify these two approaches to study the combinatorics of the two-species TASEP on a ring.

The two-species totally asymmetric simple exclusion process (2-TASEP) on a ring is a model that describes the dynamics of particles of types 0, 1, and 2 hopping clockwise around a ring of $n$ sites. Adjacent particles can swap places if the one on the left is of larger type. In the \emph{homogeneous 2-TASEP}, the swapping rates are all equal, such as in Figure \ref{TASEP_example}. We call a 2-TASEP in which each possible swap has a different rate the \emph{inhomogeneous 2-TASEP}. In this paper, we study combinatorial solutions for the stationary probabilities of the inhomogeneous 2-TASEP on a ring.

\begin{figure}[!ht]
  \centerline{\includegraphics[height=1.2in]{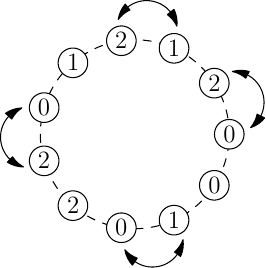}}
\centering
 \caption{A TASEP on a ring of size $(4,3,4)$. The arrows indicate possible swaps at adjacent sites.}\label{TASEP_example}
 \end{figure}

Our interest in the 2-TASEP on a ring stems from two directions. On one hand, the 2-TASEP is a specialization of the two-species asymmetric simple exclusion process (2-ASEP), in which adjacent particles swap places with rate 1 if the one on the left is of larger type, and with rate $q$ otherwise for some parameter $0\leq q\leq 1$. The 2-ASEP has recently been found to have a remarkable connection to moments of Macdonald polynomials \cite{CdGW15}. Thus in our study of combinatorics of the 2-TASEP, we hope to gain insight on the more complex 2-ASEP for which combinatorics are not yet well understood. On the other hand, the 2-TASEP is a special case of the $k$-TASEP with $k$ different types of particles, which has been studied extensively. The $k$-TASEP has a beautiful combinatorial solution in terms of \emph{multiline queues} (MLQs) discovered by Ferrari and Martin in 2005 \cite{FM07}. Our approach to solve the 2-TASEP uses tableaux, which are convenient in many ways. The tableaux are closely related to the well-studied \emph{alternative tableaux}, which solve the 2-ASEP with open boundaries. Furthermore, as we shall see, the tableaux admit a natural addition of parameters which provide a solution for the inhomogeneous 2-TASEP. 

Interest in the inhomogeneous $k$-TASEP arose from work of Lam and Williams \cite{LW12}, who studied a Markov Chain on the symmetric group, and conjectured that probabilities of this related model have a combinatorial solution consisting of polynomials with positive integer coefficients. The conjecture was proved for the 2-TASEP by Ayyer and Linusson \cite{AL14} using multiline queues, and algebraically for the $k$-TASEP by Arita and Mallick \cite{AM13}. In this paper we provide a tableaux proof which specializes to the latter.

This paper is organized as follows. In Section \ref{sec_TRAT}, we define the 2-TASEP on a ring is defined and give the solution in terms of multiline queues. We then describe the tableaux approach to get an equivalent solution. As a corollary, we obtain a determinantal formula for probabilities of states of the 2-TASEP on a ring. In Section \ref{sec_bij_main}, we give two bijections between MLQs and our tableaux. In Section \ref{sec_inhomog}, we obtain a solution to an inhomogeneous 2-TASEP that specializes to the solution in \cite{AL14}. In Section \ref{sec_TASEP_open} we extend our solution to an inhomogeneous 2-TASEP with open boundaries. Finally in Section \ref{sec_markov}, we define Markov chains on the tableaux and the MLQs that both project to the inhomogeneous 2-TASEP on a ring.

\section{The 2-TASEP and toric rhombic alternative tableaux}\label{sec_TRAT}

The 2-TASEP is a Markov chain describing particles of types 0, 1, and 2 hopping on a ring, with the larger particle types having ``priority'' over the smaller ones. The ring has $n$ sites numbered 1 through $n$ with each site occupied by one of the particle types, which we represent as a 1D periodic lattice $\mathbb{Z}/n\mathbb{Z}$. A state is represented by a word $X=X_1\ldots X_n$ where $X_i \in \{2,1,0\}$. The periodicity implies $X_1X_2\ldots X_n$ and $X_2\ldots X_n X_1$ represent the same state. We say $X$ is a state of the TASEP of size $(k,r,\ell)$ if it has $k$ 2's, $r$ 1's, and $\ell$ 0's. We denote by $\TASEP(k,r,\ell)$ the set of states of size $(k,r,\ell)$. For example, Figure \ref{TASEP_example} shows a state of $\TASEP(4,3,4)$. 

The possible transitions of the 2-TASEP chain are the following: both 2 and 1 can swap with adjacent 0's to their right. Additionally, 2 can swap with an adjacent 1 to its right:
\[X\ 2\ 0\ Y \rightarrow X\ 0\ 2\ Y, \qquad X\ 2\ 1\ Y \rightarrow X\ 1\ 2\ Y, \qquad X\ 1\ 0\ Y \rightarrow X\ 0\ 1\ Y,\]
where $X$ and $Y$ are words in $\{2,1,0\}$. In the homogeneous 2-TASEP, all transitions occur with the same rate.

The inhomogeneous multispecies TASEP has also been studied; in this model, parameters represent different hopping rates for different particle types. We will discuss the inhomogeneous 2-TASEP in Section \ref{sec_inhomog}.

A Matrix Ansatz due to Derrida, Evans, Hakim, and Pasquier gives an explicit formula for the stationary probabilities of the states of the two-species TASEP on a ring \cite{DEHP}.

\begin{defn}\label{dae}
Let $X=X_1\ldots X_n$ be a state of the 2-TASEP. For some set of matrices $D, A, E$ define $\DAE(X) = \DAE(X_1)\ldots\DAE(X_n)$ to be the matrix product given by the map $\DAE(2)\mapsto D$, $\DAE(1)\mapsto A$, and $\DAE(0)\mapsto E$. For example, $\DAE(221021)=DDAEDA$.
\end{defn}

\begin{thm}[\cite{DEHP}]\label{ansatz0}
Let $D, A, E$ be matrices that satisfy the following relations:
\begin{align*}
DE&=D+E\\
DA&=A\\
AE&=A\\
\end{align*}
Then the stationary probability of a state $X$ of the two-species TASEP on a ring  of size $(k,r,\ell)$ is given by
\[\Pr(X) = \frac{1}{{n \choose k}{n \choose \ell}} \tr(\DAE(X)),\]
where $\DAE(X)$ is given by Definition \ref{dae}.
\end{thm}

Matrices $D, A, E$ satisfying the Ansatz relations are not unique. One possible choice is:
\[
D=
\begin{pmatrix} 0&1&0&\\0&0&1&\ldots\\0&0&0&\\&\vdots&&\ddots
\end{pmatrix}
\qquad
A=
\begin{pmatrix} 1&0&0&\\1&0&0&\ldots\\1&0&0&\\&\vdots&&\ddots
\end{pmatrix}
\qquad
E=
\begin{pmatrix} 1&0&0&\\1&1&0&\ldots\\1&1&1&\\&\vdots&&\ddots
\end{pmatrix}
\]

\begin{example}
For the state $X=12011020$, $\Pr(X) = \frac{8}{{8 \choose 2}{8 \choose 3}}\tr(ADEAAEDE) = \frac{32}{{8 \choose 2}{8 \choose 3}}=\frac{1}{49}$.
\end{example}

\begin{remark}
The fact that the partition function (i.e.~normalizing factor) for the probabilities of a 2-TASEP of size $(k,r,\ell)$ is $\frac{1}{n}{n \choose k}{n \choose \ell}$ is well-known, and we will not prove it here. One way to see this is through enumeration of \emph{multiline queues}, which we discuss in the following subsection.
\end{remark}

\subsection{Multline queues}

\emph{Multiline queues} (MLQs), first introduced by Ferrari and Martin, give an elegant combinatorial formula for the stationary probabilities of the $k$-TASEP on a ring \cite{FM07}. The formula holds for any $k$, but for our purposes we define only MLQs that correspond to the 2-TASEP.

Let $k+r+\ell=n$. An MLQ of size $(k,r,l)$ is a stack of two rows of balls and vacancies with the bottom row having $r+\ell$ balls and the top row having $\ell$ balls, all within a box of size $2 \times n$. Locations are labeled from left to right with $1,\dots,n$. We identify the left and right edges of the box, making it a cylinder; thus location 1 is to the right of and adjacent to location $n$. 

Each MLQ corresponds to a state of the TASEP, which is determined by a \emph{ball drop algorithm}, consisting of balls from the top row dropping to occupy balls in the bottom row. In this algorithm, top row balls drop to the bottom row and occupy the first unoccupied bottom row ball weakly to the right. Once all the top row balls have been dropped, each occupied bottom row ball is marked as a 0-ball, and each unoccupied bottom row ball is marked as a 1-ball. A state of the TASEP is read off the bottom row by associating 0-balls, 1-balls, and vacancies to type 0, 1, and 2 particles respectively. See Figure \ref{MLQ_standard} for an example. We call this state the \emph{type} of the MLQ. For an MLQ $M$ of type $X=X_1\ldots X_n$, we denote by $M(j)$ the particle $X_j$. 

\begin{remark}
The state read off the MLQ is independent of the order in which top row balls are dropped. However, in Section \ref{sec_bij}, we will require that balls are dropped from right to left, for the purpose of our bijections.
\end{remark}

\begin{figure}[!ht]
  \centerline{\includegraphics[width=0.8\linewidth]{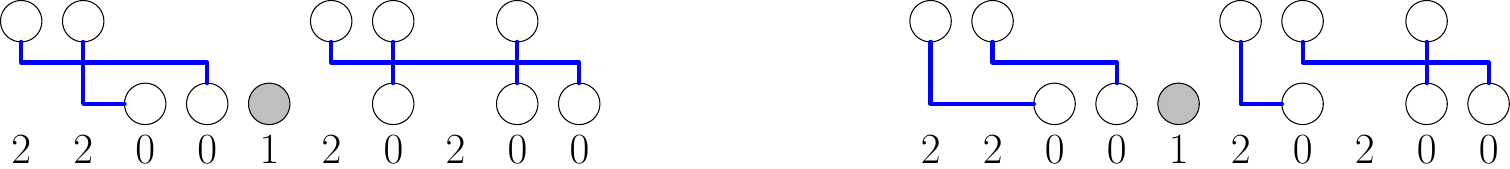}}
\centering
 \caption{Both figures represent the same MLQ of size $(4,1,5)$ and type $X=2200120200$. On the left, the top row balls are dropped from right to left, and on the right the top row balls are dropped in a different order. The occupied bottom row balls (i.e.~the 0-balls) correspond to a type 0 particle, and the unoccupied bottom row balls (i.e.~the 1-balls) correspond to a type 1 particle.}\label{MLQ_standard}
 \end{figure}

\begin{defn}
Let $\MLQ(X)$ be the set of \emph{distinct} MLQs of type $X$. By distinct, we mean that no two are cyclic shifts of each other. Figure \ref{MLQs} shows the set $\MLQ(12020)$.
\end{defn}

\begin{figure}[!ht]
  \centerline{\includegraphics[width=.8\linewidth]{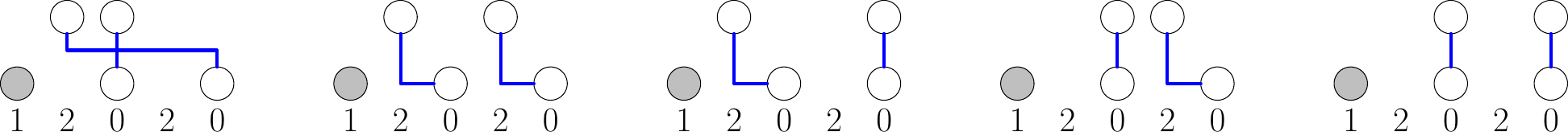}}
\centering
 \caption{The set $\MLQ(12020)$.}\label{MLQs}
 \end{figure}
 
For a formal definition, let $B=\{x_1,\ldots,x_{r+\ell}\}$ be the locations of the bottom row balls. The balls that are occupied by a dropping top row ball are in the set of locations
 \[
  H=\{x_i\ :\ \exists j\ \mbox{such that there are}\ \geq j\ \mbox{top row balls in the interval}\ [x_{i-j}+1,x_i]\}.
 \]
Then the balls in set $H$ are the 0-balls, and the balls in set $B\backslash H$ are the 1-balls, which are mapped to type 0 and type 1 particles respectively. 

Since the $\ell$ balls in the top row and the $r+\ell$ balls in the bottom row can be chosen independently, there is a total of ${n\choose\ell}{n\choose k}$ MLQs of size $(k,r,\ell)$ (where all cyclic shifts are included).

The following theorem of Ferrari and Martin gives an expression for probabilities of the 2-TASEP in terms of the MLQs. We remark that this theorem also holds for the $k$-TASEP on a ring with a more general definition of MLQs.

\begin{thm}[\cite{FM07}]\label{thm_mlq}
Let $X$ be a state of the two-species TASEP on a ring of size $(k,r,\ell)$ with $n=k+r+\ell$. Then 
\[
\Pr(X)=\frac{o(X)}{{n \choose k}{n \choose \ell}}|\MLQ(X)|,
\]
where $o(X)$ is the number of elements in the class of cyclic shifts of $X$.
\end{thm} 

\begin{example}
From Figure \ref{MLQs}, we obtain that for $X=12020$, $f(X)=5$, and so $\Pr(X)=\frac{25}{{5\choose2}{5\choose2}}=\frac{1}{4}$.
\end{example}

\subsection{Toric rhombic alternative tableaux}

In this section we introduce tableaux that we call \emph{toric rhombic alternative tableaux} (TRAT), which give a combinatorial formula for the stationary probabilities of the 2-TASEP. The TRAT are closely related to the \emph{rhombic alternative tableaux} (RAT), which were defined by the author and Viennot in \cite{MV15} as a solution for the 2-ASEP with open boundaries (see Section \ref{sec_TASEP_open}).  

The TRAT are fillings with arrows of a tiling of a closed shape whose boundary is composed of south, southwest, and west edges on a triangular lattice. The tiles are three types of rhombic tiles which we call 20-tiles, 10-tiles, and 21-tiles. Each tile can contain an arrow that points either towards its left vertical edge or the top horizontal edge; we call them \emph{left-arrows} and \emph{up-arrows} correspondingly. The rules of the filling are that any tile that is ``pointed to'' by an arrow must be empty.  We give a precise definition below.


\begin{figure}[!ht]
 \begin{minipage}{0.5\linewidth}
  \centerline{\includegraphics[height=1.5in]{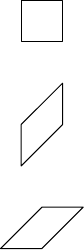}}
 \centering
  \caption{A 20-tile, a 21-tile, and a 10-tile.}
  \label{tiles}
 \end{minipage}
\hfill
\begin{minipage}{0.46\linewidth}
\centerline{\includegraphics[height=1.5in]{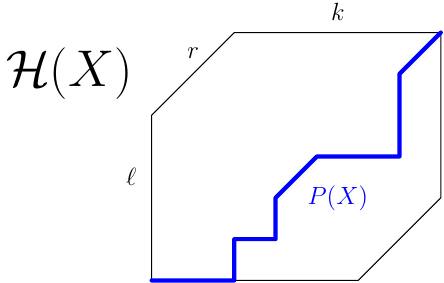}}
\centering
 \caption{Toric diagram $\mathcal{H}(X)$, 
 and the path $P(X)$.}\label{hexagon}
 \end{minipage}
\end{figure}

Let $X=X_1\ldots X_n$ with $X_i \in \{2, 1, 0\}$ be a state of size $(k,r,\ell)$ of the two-species ASEP on a ring. To guarantee the objects we introduce are well-defined, we choose a cyclic shift of $X$ such that $X_1=1$ (the reason for this will become clear later on). Define a lattice path $P(X)$ as follows: reading $X$ from left to right, draw a south edge for a 2, a southwest edge for a 1, and a west edge for a 0. Let $p_1=(\ell+r,k+r)$ and $p_4=(0,0)$ be the coordinates of the endpoints of $P(X)$. Define $p_2=(\ell+r,r)$, $p_3=(\ell,0)$, $p_5=(0,k)$, $p_6=(r,k+r)$ and let $\{p_1,\ldots,p_6\}$ be the endpoints of the diagram $\mathcal{H}$ of size $(k,r,\ell)$ that contains $P(X)$. See Figure \ref{hexagon}.

\begin{defn}
We call $\mathcal{H}(X)$ the diagram $\mathcal{H}$ together with path $P(X)$. We say $\mathcal{H}(X)$ is a \emph{toric diagram} of \emph{type} $X$. 
\end{defn}

\begin{remark}
It is certainly possible to define a toric diagram of type $X$ for $X_1=0$ or $X_1=2$ by using a different lattice path for the boundary of $\mathcal{H}(X)$. However, we have found that at this stage it suffices to limit ourselves to the definition when $X_1=1$ to simplify our presentation, with one caveat: we must take extra care in the case where $X=1Y2$ or $X=10Y$. In many of our proofs, we will give extra attention to those special cases.
\end{remark}

\begin{defn}
A \emph{20-tile} is a rhombus with south and west edges. A \emph{10-tile} is a rhombus with west and southwest edges. A \emph{21-tile} is a rhombus with south and southwest edges. See Figure \ref{tiles}. 
\end{defn}

Now choose a tiling $\mathcal{T}$ with the 20-tiles, 21-tiles, and 10-tiles on the region of $\mathcal{H}(X)$ northwest of $P(X)$ and the region of $\mathcal{H}(X)$ southwest of $P$. For the remainder of the definition of the tableaux, this tiling is fixed. We call the tiled $\mathcal{H}(X)$ a \emph{tiled toric diagram}. Figure \ref{all_strips} shows an example of a toric diagram $\mathcal{H}(X)$ of type $X=120201210$. 

\begin{defn}
A \emph{north-strip} is a connected strip composed of adjacent 20- and 10-tiles. A \emph{west-strip} is a connected strip composed of adjacent 20- and 21-tiles. The 20-tile and the 21-tile can contain a \emph{left-arrow}, which is an arrow pointing to the left vertical edge of the tile, and is also pointing to every tile to its left in its west-strip. The 20-tile and the 10-tile can contain an \emph{up-arrow}, which is an arrow pointing to the top horizontal edge of the tile, and is also pointing to every tile above it in its north strip. See Figure \ref{all_strips}.
\end{defn}

\begin{figure}[h]
  \centerline{\includegraphics[width=\linewidth]{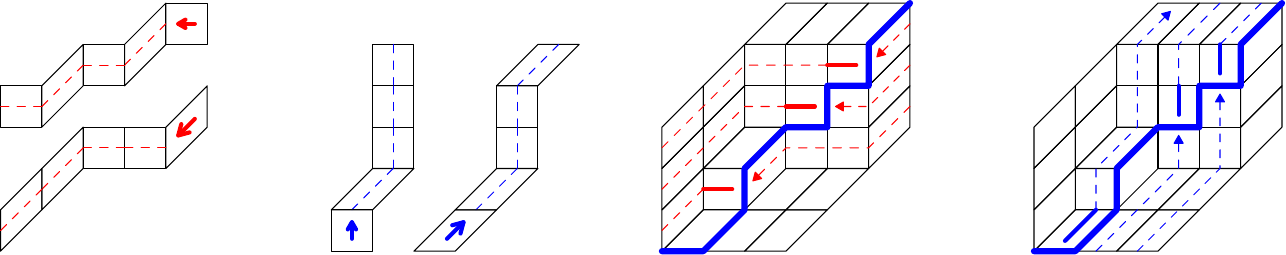}}
 \centering
  \caption{\emph{Left}: west-strips containing a left-arrow which is pointing at the tiles to its left, as well as north-strips containing an up-arrow which is pointing at the tiles above it. \emph{Middle}: a tiled toric diagram with all its west-strips highlighted. \emph{Right}: a tiled toric diagram with all its north-strips highlighted. All strips begin at the tile to the northwest and adjacent to $P(X)$ and terminate at the tile to the southeast and adjacent to $P(X)$.}\label{all_strips}
\noindent
\end{figure}

Identify the horizontal edges on the upper boundary of $\mathcal{H}(X)$ with the horizontal edges belonging to the same north-strip on the lower boundary. Similarly, identify the vertical edges on the left boundary of $\mathcal{H}(X)$ with the corresponding vertical edges belonging to the same west-strip on the right boundary. This makes $\mathcal{H}(X)$ a torus with one boundary component. If the edges of two tiles are identified, we say the tiles are adjacent. Following these identifications, north-strips and west-strips wrap around the toric diagram. Each north-strip starts at the tile directly north $P(X)$ and ends at the tile directly south of $P(X)$. Similarly, each west-strip starts at the tile directly west of $P(X)$ and ends at the tile directly east of $P(X)$, as in Figure \ref{all_strips}.

\begin{defn}
A tile is \emph{pointed at} by an arrow if it is in the same west-strip to the left of a left-arrow or if it is in the same north-strip above an up-arrow. Conversely, a tile is \emph{free} if it is not pointed at by any arrow.
\end{defn}

For consistency, we define a canonical tiling $\mathcal{T}_X$ of a toric diagram, which will be the tiling we use in most cases.

\begin{defn}\label{Tx_def}
An \emph{$X$-strip} is a north-strip obtained by reading $X$ from left to right and placing a 20-tile for a 2 and a 10-tile for a 1\ from top to bottom. Let $\mathcal{H}(X)$ be a toric diagram of size $(k,r,\ell)$. The tiling $\mathcal{T}_X$ is defined to be the top-justified placement of $\ell$ adjacent $X$-strips with $r\ell$ 21-tiles filling in the remaining space of $\mathcal{H}(X)$. For an example, see Figure \ref{Tx}.
\end{defn}

Note that we order the $X$-strips from right to left, which corresponds to locations of the 0's in $X$ from left to right.

\begin{figure}[!ht]
  \centerline{\includegraphics[height=1.2in]{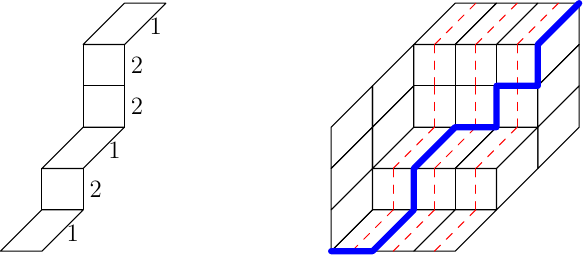}}
\centering
 \caption{For $X=120201210$, on the left is shown an $X$-strip, and on the right the tiling $\mathcal{T}_X$ on $\mathcal{H}(X)$ with path $P(X)$.}\label{Tx}
 \end{figure}

\begin{lemma}
The tiling $\mathcal{T}_X$ is a valid tiling of $\mathcal{H}(X)$ with path $P(X)$.
\end{lemma}

The lemma is easily verified with a picture, such as in Figure \ref{Tx}, but we provide the proof below.

\begin{proof}
We want to show that all the edges of $P(X)$ coincide with edges of $\mathcal{T}_X$. We obtain $P(X)$ from $\mathcal{T}_X$ as follows. 

Let $x_0=0$ and let $x_1,\ldots,x_{\ell}$ be the locations of the $\ell$ 0's in $X$ from left to right. Starting with $i=1$, from the northeast corner of $\mathcal{H}(X)$, draw a path $P$ by following the east boundary of the $i$'th $X$-strip for $x_i-x_{i-1}-1$ steps, and take a step west for the $x_i$'th step to switch to the $i+1$'st $X$-strip, up to $i=\ell$. After the $x_{\ell}$'th step, follow the west boundary of the $\ell$'th $X$-strip until the southwest corner of $\mathcal{H}(X)$ is reached.

Since the $X$-strip is obtained simply from excising the 0's from $X$, $P=P(X)$ by our construction.
\end{proof}

\begin{defn}
A TRAT of type $X$ is a filling of the tiles of a tiled toric diagram $\mathcal{H}(X)$ with left-arrows and up-arrows according to the following rules:
\begin{enumerate}
\item[i.] A tile pointed to by an arrow in the same strip must be empty.
\item[ii.] An empty tile must be pointed to by some up-arrow or some left-arrow.
\end{enumerate}
\end{defn}

Figure \ref{fillings} shows an example of all possible fillings of $\mathcal{H}(X)$ for $X=120201210$.

\begin{figure}[!ht]
  \centerline{\includegraphics[height=1in]{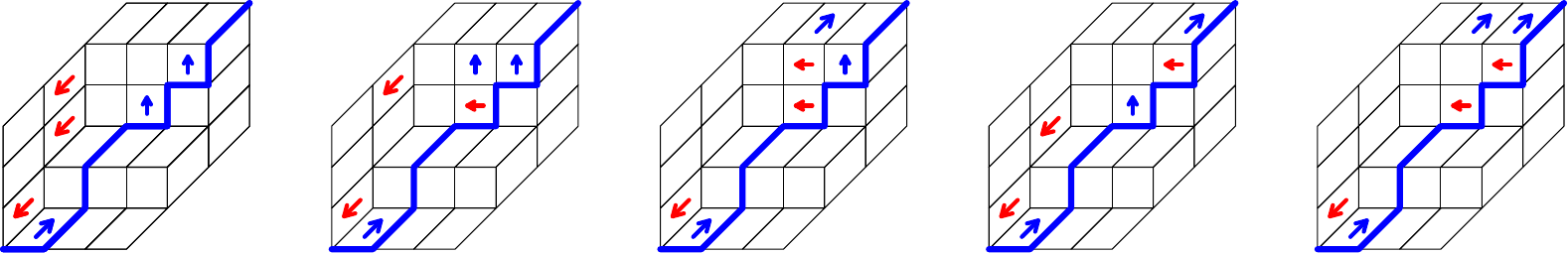}}
\centering
 \caption{The set $\TRAT(X)$ for $X=120201210$.}\label{fillings}
 \end{figure}

\begin{defn}
We define the \emph{weight} of $X$ to be the number of possible fillings of $\mathcal{H}(X)$ with tiling $\mathcal{T}$ with up-arrows and left-arrows of this tiling, and we denote it by $\wt_{\mathcal{T}}(X)$.
\end{defn}

The following lemma addresses equivalence of tilings. It is well known that any two tilings can be obtained from one another via some series of flips, where a flip is a switch of configurations in Figure \ref{flip_bij}. A \emph{filling-preserving flip} is a weight-preserving map from the filling of a tiling $\mathcal{T}$ to a filling of a tiling $\mathcal{T}'$, where $\mathcal{T}$ and $\mathcal{T}'$ differ by a single flip, with all other tiles and their contents identical in the two tilings. Figure \ref{flip_bij} shows the four possible cases of a filling-preserving flip. A full proof of this property of rhombic tableaux is given in Proposition 2.8 of \cite{MV15}.

\begin{figure}[!ht]
  \centerline{\includegraphics[height=1in]{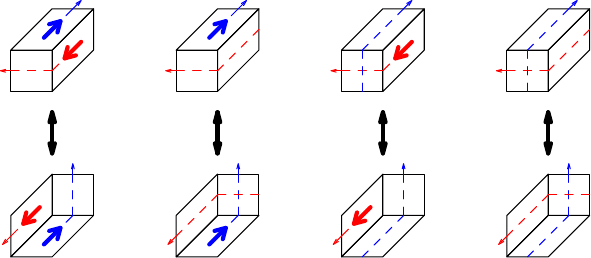}}
\centering
 \caption{A \emph{flip} is the switch from one hexagonal configuration to another in a tiling. In this figure we see the four cases of possible fillings of a hexagonal configuration of tiles in a TRAT. The dashed blue lines through the west-strips (resp.~red lines through the north strips) represent the presence of left-arrows (resp.~up-arrows) in those strips.}
 \label{flip_bij}
\end{figure}

\begin{lemma}\label{flip_lemma}
Let $\mathcal{T}$ and $\mathcal{T}'$ be two different tilings on $\mathcal{H}(X)$. Then
\[ \wt_{\mathcal{T}}(X)=\wt_{\mathcal{T}'}(X).
\]
\end{lemma}

As a consequence of the Lemma, we are able to define the weight of a state $X$.
\begin{defn}
Choose any tiling $\mathcal{T}$ on $\mathcal{H}(X)$ of size $(k,r,\ell)$. Define 
\[\weight(X) = \wt_{\mathcal{T}}(X).\]
\end{defn}

\begin{defn}
Let $X\in\TASEP(k,r,\ell)$. We denote by $o(X)$ the \emph{order} of $X$, which is the number of elements in the class of cyclic shifts of $X$.
\end{defn}

Our main result is the following.

\begin{thm}\label{TRAT_thm}
Let $X$ be a state of the two-species TASEP on a ring of size $(k,r,\ell)$ with $n=k+r+\ell$. Then 
\[
\Pr(X) =\frac{o(X)}{{n \choose k}{n \choose \ell}} \weight(X).
\]
\end{thm}

\begin{example}
From Figure \ref{MLQs}, we obtain that for $X=120201210$, $f(X)=5$, and so $\Pr(X)=\frac{9}{{9\choose3}{9\choose3}}\cdot 5 = \frac{5}{748}$ since there is a total of ${9\choose3}{9\choose3}$ MLQs of size $(3,3,3)$ and $o(X)=9$. On the other hand, for $Y=201201201$, $f(Y)=8$, $o(Y)=3$, and so $\Pr(Y)=\frac{1}{294}$.
\end{example}

\begin{remark}
Note that in most cases, $o(X)=n$, unless $k, \ell$, and $r$ have a common factor. When $\gcd(k,\ell,r)=1$, we can write 
\[\Pr(X) =\frac{n}{{n \choose k}{n \choose \ell}} \weight(X).
\]
\end{remark}

We will first give a canonical Matrix Ansatz proof below, and in Section \ref{sec_bij}, we will show the TRAT is in bijection with the MLQs, from which our theorem follows due to Theorem \ref{thm_mlq}. 

We prove Theorem \ref{TRAT_thm} by showing by induction on $|X|$ that $\weight(X)$ satisfies the same recurrences as the Matrix Ansatz of Theorem \ref{ansatz0}.

\begin{proof}
When $X$ contains zero type 1 particles, an exceptional case occurs, since the trace of the matrix product of matrices $D$ and $E$ is no longer finite, so we cannot use the standard Matrix Ansatz proof. For $X$ of size $(k,0, \ell)$, $\mathcal{H}(X)$ is a $k \times \ell$ rectangle, and $\weight(X)={k+\ell \choose k}$. (This can be derived with a standard lattice path bijection in the flavor of the Catalan tableaux that appear in \cite{Vie07}, which we will not expand upon here.) It is also easy to check that the TASEP on a ring with fewer than two species of particles has uniform stationary distribution (each state has the same number of outgoing transitions as it has incoming transitions, so detailed balance holds). There is a total of ${k+\ell \choose k}^2$ MLQs and ${k+\ell \choose k}$ total states counting all cyclic shifts, and so $\Pr(X)=\frac{o(X)}{{k+\ell \choose k}}$ since $X$ is counted $o(X)$ times. Consequently, Theorem \ref{TRAT_thm} trivially holds in this case.

Let $f(X)=\tr(\DAE(X))$, as defined in Theorem \ref{ansatz0}. When $X$ has at least one type 1 particle, we will show that $\weight(X)=f(X)$. Our proof is by induction on $|X|$.

For the base cases, when $X$ has size $(k,r,0)$ or $(0,r,\ell)$, $\mathcal{H}(X)$ consists of only 21-tiles or 10-tiles respectively, and so in each case, there is a unique filling of $\mathcal{H(X)}$. Thus since $DA=AE=A$, we trivially obtain $1=\weight(X)=f(X)$. Now, let $n$ be such that for any $|W|<n$, it holds that $\weight(W)=f(W)$. 

Let $X$ have size $(k,r,\ell)$ with $k,r,\ell>0$ and $k+r+\ell=n$. Assume $X=1Y$ for some $Y$. One of the following must occur: 

\textbf{Case 1.} $X=Y'20Y''$,

\textbf{Case 2.} $X=Y'21Y''$, or

\textbf{Case 3.} $X=1Y'2$.

We fix the tiling $\mathcal{T}_X$ on $\mathcal{H}(X)$ and consider each of these cases.

\begin{figure}[!ht]
  \centerline{\includegraphics[height=2.5in]{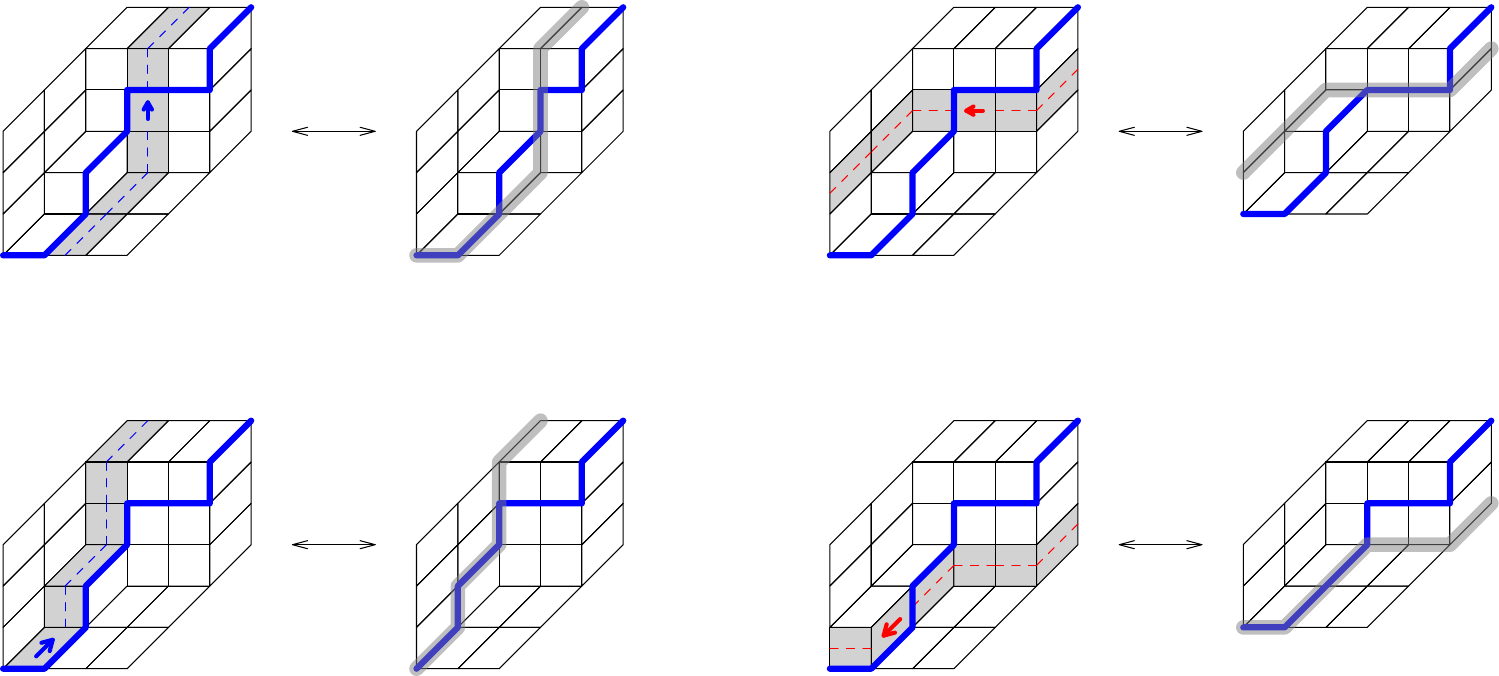}}
\centering
 \caption{\emph{Top left}: 20-tile containing an up-arrow. \emph{Top right}: 20-tile containing a left-arrow. \emph{Bottom left}: 10-tile containing an up-arrow. \emph{Bottom right}: 21-tile containing a left-arrow. For each example, the fillings of the smaller tableau are in bijection with fillings of the tableau with the highlighted strip removed (the pink path in the smaller tableau indicates the location of the removed strip).}\label{MA_recurrences}
 \end{figure}

\emph{Case 1: $X=Y'20Y''$.} The 20-tile adjacent to the 2, 0 pair of edges of $P(X)$ necessarily contains either a left-arrow or an up-arrow. In the left-arrow case, the remaining tiles of the west-strip $\textbf{w}$ originating at the 2-edge must be empty. Then the fillings of $\mathcal{H}(X)$ are in bijection with the fillings of $\mathcal{H}(X)\backslash\textbf{w}$ which is a tiled rhombic diagram of shape $Y'0Y''$. In the up-arrow case, the remaining tiles of the north-strip $\textbf{n}$ originating at the 0-edge must be empty. Then the fillings of $\mathcal{H}(X)$ are in bijection with the fillings of $\mathcal{H}(X)\backslash\textbf{n}$ which is a tiled rhombic diagram of shape $Y'2Y''$. 

Figure \ref{MA_recurrences} illustrates both of these cases.

Consequently,
\[
\weight(Y'20Y'')=\weight(Y'2Y'')+\weight(Y'0Y'') = f(Y'2Y'')+f(Y'0Y'') = f(Y'20Y'')
\]
by the inductive hypothesis since $|Y'0Y''|<n$ and $|Y'2Y''|<n$, and hence we obtain the desired result.

\emph{Case 2: $X=Y'21Y''$.} The 21-tile adjacent to the 2, 1 pair of edges of $P(X)$ necessarily contains a left-arrow. Thus the remaining tiles of the west-strip $\textbf{w}$ originating at the 2-edge must be empty. Hence the fillings of $\mathcal{H}(X)$ are in bijection with the fillings of $\mathcal{H}(X)\backslash\textbf{w}$ which is a tiled rhombic diagram of shape $Y'1Y''$. Consequently,
\[
\weight(Y'21Y'') =  \weight(Y'1Y'')=f(Y'1Y'')=f(Y'21Y'')
\]
by the inductive hypothesis since $|Y'1Y''|<n$. Thus $\weight(X)=f(X)$, as desired.

\emph{Case 3: $X=1Y'2$.} This case is interesting since a toric diagram of type $X$ by construction has no 21-tile adjacent to the 2, 1 pair of edges at the ends of $P(X)$, so we cannot perform the simple recurrence of the first two cases. By our convention, $X$ is required to start with a 1, but fortunately this case is quite simple. We consider the bottom-most west-strip $\textbf{w}$ corresponding to the 2-edge, a in Figure \ref{MA_rec_exception}. The rightmost tile of $\textbf{w}$ is a 21-tile, and hence it must contain a left-arrow with its remaining tiles empty; this means $\textbf{w}$ is completely independent from the rest of the tableau. It is immediate that the fillings of $\mathcal{H}(X)$ are in bijection with the fillings of $\mathcal{H}(X)\backslash\textbf{w}$, which is a tiled rhombic diagram of shape $1Y'$. Consequently,
\[
\weight(1Y'2) =  \weight(1Y')=f(1Y')=f(1Y'2)
\]
by the inductive hypothesis since $|1Y'|<n$. Thus $\weight(X)=f(X)$ in all three cases, and our proof is complete.
\end{proof}

\begin{figure}[!ht]
  \centerline{\includegraphics[height=1.3in]{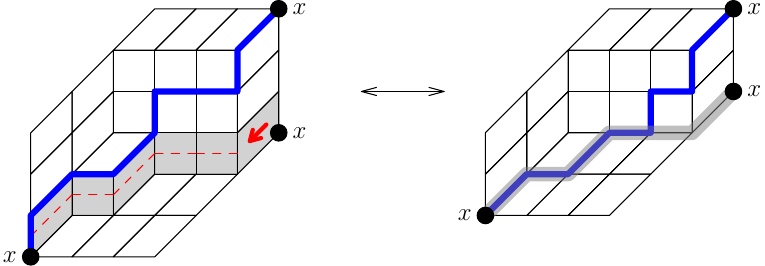}}
\centering
 \caption{A tableau corresponding to $X=1Y2$, where we consider the transition $21 \rightarrow 12$ of the boundary edges of $P(X)$. All red points labeled $x$ represent the same point. The fillings of the smaller tableaux are in bijection with fillings of tableaux with the highlighted strip removed (the pink path in the smaller tableaux indicates the location of the removed strip).}\label{MA_rec_exception}
 \end{figure}

Therefore, the TRAT indeed provide combinatorial formulae for the probabilities of the two-species TASEP on a ring.

\subsection{Determinantal formula for probabilities of the two-species TASEP on a ring}

We use the results of \cite{Man15} to compute $\weight(X)$ using a determinantal formula that arises from the non-crossing paths Lingstr\"om-Gessel-Viennot Lemma.

Call an interval of of $X$ consisting of 0 and 2 particles a \emph{0,2-interval}. Partition $X$ into $r$ maximal 0,2-intervals $X^1,\ldots, X^r$.

\begin{defn}
Let $X \in \{0,2\}^{j+m}$ and let there be $j$ 2's at locations $a_1,\ldots,a_j$. Define $\lambda(X)$ to be the partition associated to the Young diagram whose southeast boundary coincides with $P(X)$. Namely,
\[
\lambda(X)=(m+1-a_1,m+2-a_2,\ldots,m+j-a_j).
\]
\end{defn}

For an example, see Figure \ref{lambda}.

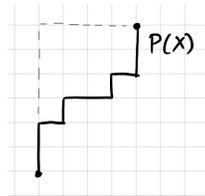
\begin{figure}[!ht]
 \begin{tikzpicture}[scale=0.5]
\draw[help lines,step=1] (4,0) grid (0,-6);
 \draw[ultra thick,blue] (4,0)--(4,-2)--(3,-2)--(3,-3)--(1,-3)--(1,-4)--(0,-4)--(0,-6);
\draw (5.5,-1) node {$P(X)$};
 \end{tikzpicture}
\centering
 \caption{The Young diagram associated to 0,2-word $X=2202002022$. Here $(a_1,\ldots,a_6)=(1,2,4,7,9,10)$, and $\lambda(X)=(4,4,3,1,0,0)$.}\label{lambda}
 \end{figure}
 
 The following is derived in \cite{Man15,Man16}. For a partition $\lambda$, define 
 \[
A_{\lambda} = \left( {\lambda_j+1 \choose j-i+1} \right)_{(i,j)}
\]
 
 \begin{thm}
 Let $X$ be a state of the two-species ASEP on a ring. Partition $X$ into 0,2-intervals $X_1,\ldots, X_{r+1}$. Then
 \[
 \weight(X) = \prod_{i=1}^{r+1} \det A_{\lambda(X_i)}.
 \]
 \end{thm}

\section{Bijections}\label{sec_bij_main}

In this section, we describe two different bijections between MLQs and TRAT; one weight preserving and one not. The first bijection relies on a particular order for the ball drop algorithm on the MLQs which we discuss in the following subsection. In section \ref{sec_weights}, this weight-preserving bijection will permit us to define weighted multiline queues that give a combinatorial solution for the inhomogeneous TASEP. For the second bijection, the order of ball drops does not matter; we are still able to define weights on the MLQs, but the bijection with TRAT is no longer weight-preserving.

\subsection{Refined multiline queue definition}\label{mlq_def}

Each multiline queue corresponds to a state of the circular ASEP, determined by the (order independent) \emph{ball dropping algorithm} given in Section \ref{sec_intro}. We label the bottom row balls as 0-balls (balls occupied by a top row ball) and 1-balls (unoccupied balls).

To make our bijection well-defined, we first cyclically shift the MLQ to have a 1-ball at the left-most bottom row location. This implies no top row ball will wrap around the MLQ when it drops. Let the bottom row 0-balls be in locations $(x_1,\ldots,x_{\ell})$. Now, drop the top row balls from right to left. With each drop, the ball occupies the first unoccupied bottom row ball weakly to its right, while marking unmarked bottom row vacancies. 

Let $w_i$ be the number of unmarked vacancies that were marked by the dropping ball that occupied the 0-ball at location $x_i$. Set $w_i$ to be the \emph{weight} of $x_i$. Figure \ref{mlq_ws} shows an example with weights $(1,1,0,0,2,0,1)$.

\begin{figure}[!ht]
  \centerline{\includegraphics[height=1in]{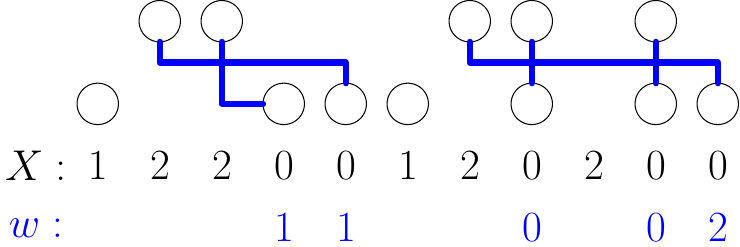}}
\centering
 \caption{For $X=22001202001020$, the hitting weights of the ball drops are $(1,1,0,0,2,0,1)$.}\label{mlq_ws}
 \end{figure}

\begin{lemma}\label{uniqueness}
At the end of the ball drop algorithm, the list of weights $(w_1,\ldots,w_{\ell})$ uniquely determines the initial configuration of top row balls.
\end{lemma}

The lemma is proved simply, by reversing the ball drop algorithm and ``lifting'' the bottom row 0-balls from right to left such that each ball at location $x_i$ marks $w_i$ unmarked vacancies. We call the reverse of a ball drop to a bottom row 0-ball at location $x_i$ a \emph{ball lift} from the 0-ball at location $x_i$, defined below.  

\begin{defn}\label{ball_lift}
Let $(x_1,\ldots,x_{\ell})$ be the locations of the 0-balls of an MLQ with corresponding weights $(w_1,\ldots,w_{\ell})$; all vacancies are initially unmarked. A \emph{ball lift} from a 0-ball at location $x_i$ with weight $w_i$ is the following. A top row ball is placed directly above the $w_i$'th consecutive unmarked vacancy to the left of $x_i$, and each of those $w_i$ vacancies becomes marked. 
\end{defn}

To show the ball lift is well-defined, i.e. that there are always $w_i$ unmarked vacancies to the left of $x_i$ with no 1-ball in between, we need the following lemma, the proof of which is obtained directly by following the ball-drop algorithm.

\begin{lemma}\label{MLQ_technical}
Suppose $M$ is an MLQ of size $(k,r,\ell)$, and in the bottom row, $x_1<\cdots<x_{\ell}$ are the locations of the 0-balls, and $b_1\leq \cdots \leq b_{\ell}$ are the locations of the \emph{nearest} 1-balls, defined by
\[
 b_i = \max\{y<x_i\ :\ \mbox{1-ball at location } y\}. 
\]
Let $(w_1,\ldots,w_{\ell})$ be the weights on locations $(x_1,\ldots,x_{\ell})$ after the ball drops. Then the conditions on $(w_1,\ldots,w_{\ell})$ are:
\[
 \sum_{j:\ b_i <x_j \leq x_i} w_j+1 \leq x_i-b_i
\]
for each $i$. In other words, there are enough vacancies to the left of $x_i$ so that it can have weight $w_i$. We call such a list $(w_1,\ldots,w_{\ell})$ an \emph{$X$-consistent} list.
\end{lemma}

\begin{proof}[Proof of Lemma \ref{uniqueness}]
The lemma is equivalent to showing that $M$ is the unique MLQ of type $X$ with $X$-consistent weights $(w_1,\ldots,w_{\ell})$. We show this by reconstructing an MLQ of type $X$ from an $X$-consistent list $(w_1,\ldots,w_{\ell})$. Let $X$ have its 0 particles at locations $(x_1,\ldots,x_{\ell})$. Now perform ball lifts on the 0-balls from right to left (which is precisely the reverse of the ball-drop algorithm). For a ball lift with weight $w_i$ to be possible, there must be at least $w_i$ unmarked vacancies to the left of $x_i$ with no 1-balls in between. This translates precisely to the requirement that 
\[
 x_i-b_i-1- \sum_{j:\ b_i <x_j < x_i} w_j+1 \geq w_i,
\]
which we notice is the same as the condition placed on the $w_i$'s in Lemma \ref{MLQ_technical}. Thus the ball drop algorithm has a well-defined inverse, and so the $X$-consistent list of weights $(w_1,\ldots,w_{\ell})$ corresponds to a unique MLQ of type $X$.
\end{proof}


\subsection{Map from TRAT to MLQ}\label{sec_bij}

Let $R$ be a TRAT of type $X\in \TASEP(k,r,\ell)$. To describe a well-defined map $R$ to a multiline queue, we first perform flips on the tiling of $R$ to obtain the tiling $\mathcal{T}_X$ from Definition \ref{Tx_def}.

Without loss of generality, let $X$ begin with a 1. Note that any north-strip that does not have an up-arrow below a 10-tile will necessarily acquire an up-arrow at the 10-tile. Thus there can be no arrows in north-strips above any 10-tiles. In particular, this implies the TRAT $R$ of type $X$ has all of the left-arrows and up-arrows contained in its $\ell$ $X$-strips above the path $P(X)$ in $\mathcal{H}(X)$, and so there is no ambiguity about which strip to start with. 

We build an MLQ $\mlq(R)$ from $R$ as follows. Let the bottom row of $\mlq(R)$ have type $X$. Let $R$ have its north-strips at locations $x_1<\cdots<x_{\ell}$ (from right to left) with $a_i$ left-arrows in strip $x_i$ for each $i$. Perform ball-lifts (of Definition \ref{ball_lift}) sequentially for $x_1,\ldots,x_{\ell}$, with weights $a_1,\ldots,a_{\ell}$, to obtain a unique MLQ $\mlq(R)$ with those weights. Figure \ref{bijection_example} shows an example, and the following lemma shows our bijection is well-defined.

\begin{figure}[!ht]
  \centerline{\includegraphics[height=1in]{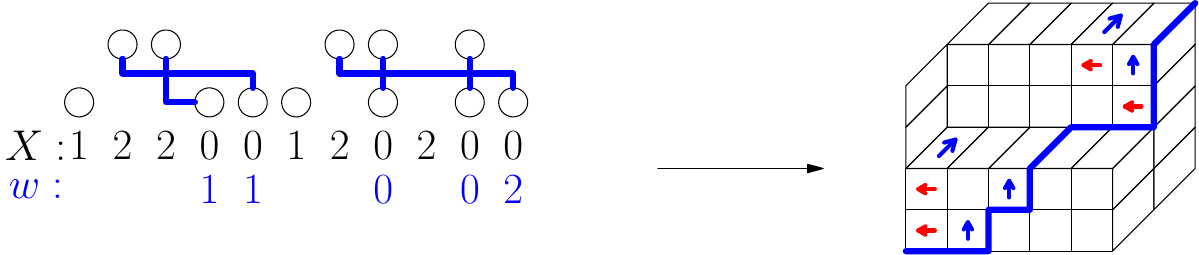}}
\centering
 \caption{The weights of the ball drops of the MLQ on the left are $(1, 1, 0, 1, 1, 0)$, which is also the number of left-arrows in each north-strip from right to left in the corresponding TRAT on the right.}\label{bijection_example}
 \end{figure}

\begin{lemma}
 Suppose the north-strips of $R$ are at locations $x_1<\cdots<x_{\ell}$, with strip $x_i$ containing $a_i$ left-arrows for each $i$. Then $(a_1,\ldots,a_{\ell})$ is an $X$-consistent list, and thus there exists a unique MLQ of type $X$ with weights $(a_1,\ldots,a_{\ell})$.
\end{lemma}

Recall that a free 20-tile is one that does not have a left-arrow to its right in the same west-strip or an up-arrow below in the same north-strip.

\begin{proof}
Our proof will show that there is a natural map between the number of left-arrows in north-strip $x_i$ in $R$ and the weight $w_i$ of the 0-ball at location $x_i$ in $\mlq(R)$. 

If a north-strip at location $x_i$ contains $a_i$ left-arrows, then it must have at least $a_i$ free 20-tiles below its first 10-tile. Let $b_i$ be the index of the diagonal strip containing the nearest 10-tile in strip $x_i$. Since diagonal strips cannot intersect, $b_i$ is the index of the nearest diagonal edge to the right of $x_i$ in $P(X)$. In other words, $b_i = \max\{y\ :\ y<x_i,\ X_y=1\}$.\footnote{We note here that this definition of $b_i$ is precisely the location of the first 10-tile only when the particular tiling $\mathcal{T}_X$ is used. That is because the order of the 2- and 1-edges in $P(X)$ matches the order of the 20- and 10-tiles in the $x_i$ north-strip.} Then we have the following conditions on the $a_i$'s.
For each $i$,
\[
 \sum_{j:\ a_i <x_j \leq x_i} a_j+1 \leq x_i-b_i.
\]
Observe that the conditions on the list $(a_1,\ldots,a_{\ell})$ make it an $X$-consistent list. Thus by Lemma \ref{uniqueness}, there exists a unique MLQ $M(R)$ of type $X$ with weights $(a_1,\ldots,a_{\ell})$, obtained by performing ball lifts sequentially for $x_1,\ldots,x_{\ell}$. This completes our proof.
\end{proof}

The inverse map from $\trat: \MLQ(k,r,\ell) \rightarrow \TRAT(k,r,\ell)$ is obtained similarly. Let $M$ have type $X$ with the 0-balls at locations $x_1<\cdots<x_{\ell}$, and with corresponding weights $(w_1,\ldots,w_{\ell})$. Construct a TRAT with shape $\mathcal{H}(X)$ with tiling $\mathcal{T}_X$ such that strip $x_i$ has $w_i$ left-arrows for each $i$ (strips are ordered from right to left). This construction is well-defined since the list $(w_1,\ldots,w_{\ell})$ is $X$-consistent, which is a sufficient condition for a TRAT with such properties to exist. It is unique by construction: when filling the tableau from right to left, in each north-strip the left-arrows must be placed in consecutive free tiles from bottom to top. The maps $R \rightarrow \mlq(R)$ and $M \rightarrow \trat(M)$ are immediately inverses of each other.

\subsection{Nested path map from MLQ to TRAT}\label{sec_paths}
Using nested lattice paths, we obtain a different bijection from MLQs to TRAT. In the following, we assume the MLQ has a 1-ball at its leftmost bottom row location.

A multiline queue naturally has an interpretation in terms of weighted lattice paths, where each row of the MLQ is mapped to a path, and each location in the row determines the type of edge that appears in the path. At $q=0$, fillings of RAT are in bijection with weighted nested lattice paths in the usual 2-TASEP \cite{Man16}, and indeed this property is preserved in the case of the TRAT. Unfortunately, the pairs of lattice paths corresponding to the MLQs are not the same paths that are in bijection with the TRAT. That is, the pair of nested paths corresponding to the TRAT $\trat(M)$ is not the same as the pair of nested paths directly obtained from $M$. However, the set of paths of type $X$ that is obtained from MLQs of type $X$ is the same as the set of paths of type $X$ obtained from TRAT of type $X$; see, for example, Figure \ref{path_to_TRAT}.

\begin{figure}[!ht]
  \centerline{\includegraphics[width=\linewidth]{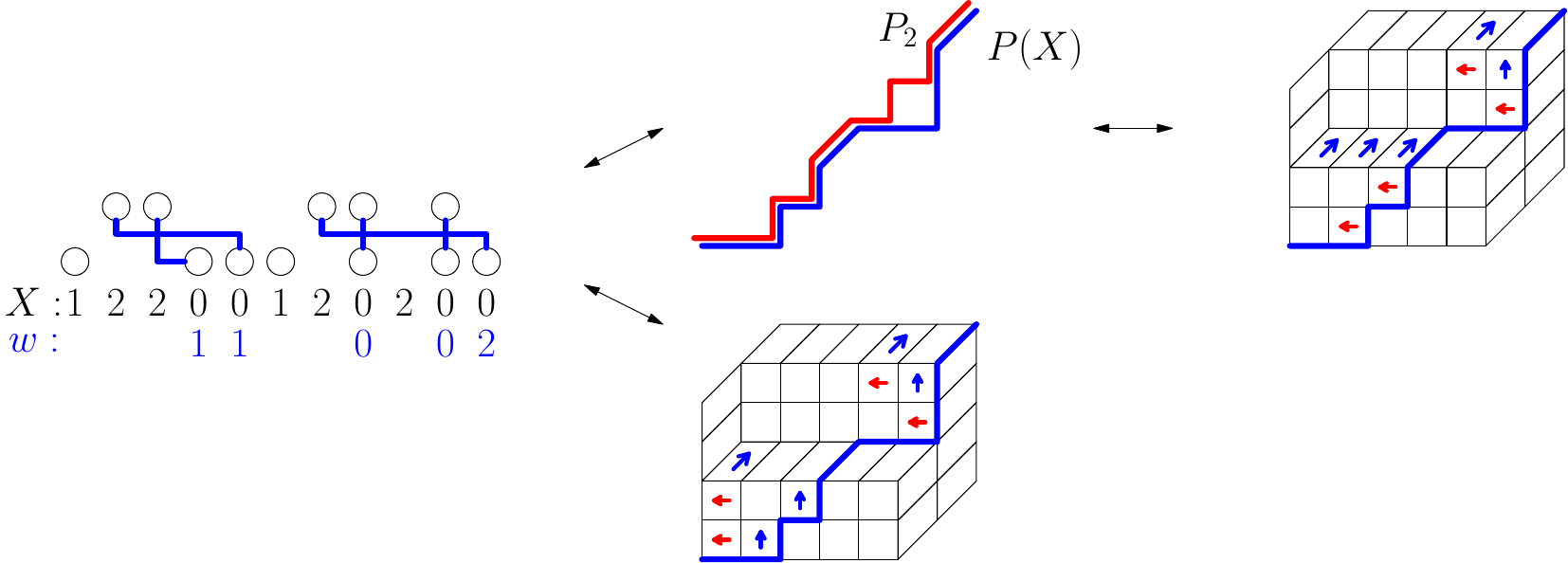}}
\centering
 \caption{The map from a MLQ to a nested pair of lattice paths, which then maps to a TRAT via the canonical lattice path bijection. The second map is from the MLQ $M$ to the TRAT $\trat(M)$. Notice that the two resulting TRAT are not the same.}\label{path_to_TRAT}
 \end{figure}

\begin{defn}
A \emph{2-TASEP compatible} pair of lattice paths is a pair of paths composed of south, west, and southwest edges that coincide at their endpoints, such that the space between the two paths can be completely tiled by squares.
\end{defn}

We construct a pair of lattice paths from an MLQ as follows: the first path, $P_1$, is obtained by reading the bottom row of the MLQ and drawing a south edge for every vacancy, a west edge for every 0-ball, and a southwest edge for every 1-ball. The second path, $P_2$, is obtained by reading the top row of the MLQ and drawing a southwest edge for a vacancy directly above a bottom row 1-ball, a south edge for a vacancy otherwise, and a west edge for a ball.

\begin{lemma}\label{lem_paths}
By our construction, $P_2$ is weakly above $P_1$, and they coincide at every diagonal edge.
\end{lemma}

\begin{proof}
We consider an interval of vacancies and 0-balls between any two 1-balls at locations $a$ and $b$ in the bottom row of $M$. At each location $a<j\leq b$, there must be at least as many top row balls as there are bottom row 0-balls between locations $a$ and $j$; otherwise, there will be an unoccupied 0-ball, which is a contradiction. Moreover, between $a$ and $b$, there must be exactly the same number of top row balls as there are bottom row 0-balls. The latter implies $P_2$ and $P_1$ coincide at every diagonal edge. Hence for each $a<j\leq b$, $P_1$ takes at least as many steps south as does $P_2$ between edges $a$ and $j$, and thus $P_2$ lies weakly above $P_1$.
\end{proof}

\begin{lemma} $P_1$ and $P_2$ are 2-TASEP compatible paths.
\end{lemma}

\begin{proof}
By Lemma \ref{lem_paths}, the space between the two paths is always bounded by horizontal and vertical edges, and thus can be tiled completely by 20-tiles.
\end{proof}

It is easy to see that any pair of 2-TASEP compatible lattice paths corresponds to a unique multiline queue and vice versa. Building on the author's earlier paper \cite{Man16}, 2-TASEP compatible lattice paths are also in bijection with TRAT. This bijection arises from the canonical lattice path bijection of Catalan paths and Catalan tableaux in the well-studied case of the usual TASEP \cite{Vie07}. We describe the map briefly.
A path weakly above the path $P(X)$ for a TRAT of shape $\mathcal{H}(X)$ is constructed as follows.

The path $P_2$ begins and ends at the endpoints of $P(X)$. It contains $n$ edges, $r$ of which are diagonal, $k$ of which are vertical, and $\ell$ of which are horizontal; its edges are labeled from right to left. Suppose $X$ has its type 1 particles at locations $b_1,\ldots,b_r$. Then $P_2$ has its diagonal edges at locations $b_1,\ldots,b_r$. At each 0-edge of $P(X)$, the path $P_2$ takes $j$ vertical steps down and one horizontal step left, where $j$ is the total number of left-arrows in the 20-tiles of that 0-strip. Once $P_2$ has reached the left border of $\mathcal{H}(X)$, it takes vertical steps down until it reaches the left endpoint of $P(X)$. See an example in Figure \ref{TRAT_to_path}.

\begin{figure}[!ht]
  \centerline{\includegraphics[height=1in]{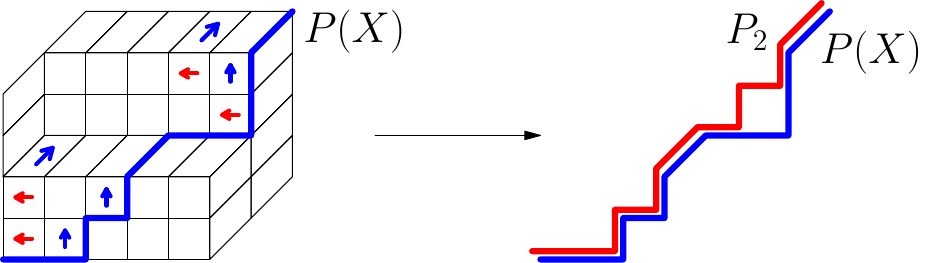}}
\centering
 \caption{An example of the canonical map from a TRAT to a pair of nested lattice paths of type $1220201100$.}\label{TRAT_to_path}
 \end{figure}

The reverse map is as follows: begin with a pair of 2-ASEP compatible nested paths $P(X)$ and $P_2$, assuming $X$ begins with a type 1 particle. Starting from right to left, let each 0-strip of the filling of $\mathcal{H}(X)$ contain $j$ left-arrows in its 20-boxes, followed by an up-arrow, where $j$ is the number of down-steps in $P_2$ preceding the horizontal step corresponding to the given 0-strip. There is a unique way of filling this 0-strip in such a way: the left-arrows must be in the lowest tiles possible, immediately followed by the up-arrow. See an example in Figure \ref{path_to_TRAT}.

This bijection is well-defined due to the following lemma.

\begin{lemma}
2-TASEP compatible paths of type $X$ are in one to one correspondence with an $X$-consistent list.
\end{lemma}

\begin{proof}
Let $x_1<\cdots<x_{\ell}$ be the indices of the 0 particles in $X$. Let $a_i$ be the number of south steps in $P_2$ on the right of the $x_i$-column. Let $b_i$ be the index of the nearest 1-particle to the left of $x_i$. $P_2$ is always weakly above $P_1$, and there can never be more south steps in $P_2$ than there are south steps in $P_1$ in the same interval. Thus 
\[
\sum_{j:\ x_i \geq x_j>b_i} a_j+1 \leq x_i-b_i,
\]
which is precisely the condition for $(a_1,\ldots,a_{\ell})$ to be an $X$-consistent list.

On the other hand, if $(a_1,\ldots,a_{\ell})$ satisfies the equation above, we have that at every $x_i$ column, there are at least $a_i$ possible south steps $P_2$ can take so that it is still weakly above $P_1$. Thus $P_2$ and $P_1$ with the given $X$-consistent list of south steps are indeed 2-TASEP compatible paths.
\end{proof}

\section{Inhomogeneous 2-TASEP on a ring}\label{sec_inhomog}

We define the inhomogeneous 2-TASEP on a ring Markov chain as follows: let $X\in \TASEP(k,r,\ell)$. The transitions on this Markov chain are:
\begin{align*}
X'20X''& \mathrel{\mathop{\rightarrow}^{\mathrm{t}}} X'02X''\\
X'21X''& \mathrel{\mathop{\rightarrow}^{\mathrm{d}}} X'12X''\\
X'10X''& \mathrel{\mathop{\rightarrow}^{\mathrm{e}}} X'01X''
\end{align*}
where $0 \leq t,d,e \leq 1$ are parameters describing the hopping
rates. When $t=d=e=1$, we recover the usual 2-TASEP on a ring. When
$t=e$, we recover the inhomogeneous TASEP studied by Ayyer and
Linusson in \cite{AL14}, where they defined weights on MLQs
to solve a conjecture of Lam and Williams \cite{LW12} (our
solution specializes to theirs after some manipulation). The advantage
of our tableaux interpretation of 2-TASEP probabilities is that we can
introduce additional weights to the TRAT which correspond to an inhomogeneous 2-TASEP. Define $\wt(X)$ to be the unnormalized steady state probability of state $X$. We will show it is a polynomial in $t,d,e$ with coefficients in $\mathbb{Z}^+$ by expressing it as a sum over the weighted tableaux.

The Matrix Ansatz of Theorem \ref{ansatz0} naturally generalizes to
the following inhomogeneous version.
\begin{thm}\label{inhomog_ansatz_thm}
Let $X$ be a state of the inhomogeneous 2-TASEP. Let $D$, $A$, and $E$ be matrices satisfying:
\begin{align}\label{inhomog_ansatz}
tDE&=D+E\\\nonumber
dDA&=A\\\nonumber
eAE&=A
\end{align}
then the stationary probability $\Pr(X)$ is proportional to $\tr(\DAE(X))$, where $\DAE(X)$ is given by Definition \ref{dae}.
\end{thm}

A set of matrices that satisfy the conditions of the Ansatz are:
\[
D^*=
\begin{pmatrix} 0&\frac{1}{d}&0&0&\\0&0&\frac{1}{d}&0&\ldots\\0&0&0&\frac{1}{d}&&\\0&0&0&0\\&\vdots&&&\ddots
\end{pmatrix}
\quad
A^*=
\begin{pmatrix} 1&0&0&0&\\1&0&0&0&\ldots\\1&0&0&0&\\1&0&0&0&&\\&\vdots&&&\ddots
\end{pmatrix}
\quad
E^*=
\begin{pmatrix} \frac{1}{e}&0&0&0&\\\frac{d}{te}&\frac{1}{t}&0&0&\ldots\\\frac{d^2}{et^2}&\frac{d}{t^2}&\frac{1}{t}&0&\\\frac{d^3}{et^3}&\frac{d^2}{et^2}&\frac{d}{et}&\frac{1}{t}&&\\&\vdots&&&\ddots
\end{pmatrix}
\]

\begin{example} For example, $\weight(2201021)=\tr(D^*D^*E^*A^*E^*D^*A^*)=\frac{1}{e^2t^2d^3}(d^2+de+te)$ where $D^*, A^*, E^*$ satisfy Equations \ref{inhomog_ansatz}.
\end{example}

\begin{defn}
For $X\in\TASEP(k,r,\ell)$, we call $\TRAT_{\mathcal{T}}(X)$ the set of TRAT on a toric diagram of type $X$ with some fixed tiling $\mathcal{T}$.
\end{defn}

We introduce a weight on the TRAT, which is a monomial in $d$, $e$, and $t$, and is denoted by $\wt(R)$ for $R\in\TRAT_{\mathcal{T}}(X)$ for some tiling $\mathcal{T}$. We define $\wt(X)=\sum_{R\in\TRAT_{\mathcal{T}}(X)}\wt(R)$, which we will show satisfies the same recurrences as $\tr(\DAE(X))$ in the Matrix Ansatz. Given the existence of $D^*, A^*, E^*$ above, we will thus obtain that $\Pr(X)$ is proportional to $\wt(X)$.

\begin{defn}
Let $R$ be a TRAT. Define $\Left(R)$ to be the number of 20-tiles in $R$ containing a left-arrow, and $\Up(R)$ to be the number of 20-tiles in $R$ containing an up-arrow.
\end{defn}

\begin{defn}
Let $X\in\TASEP(k,r,\ell)$ and $R\in\TRAT_{\mathcal{T}}(X)$ for some tiling $\mathcal{T}$. The \emph{weight} of $R$ denoted by $\wt(R)$, is given by 
 \[
   \wt(R)=d^{\Left(R)}e^{\Up(R)}t^{k+\ell-\Left(R)-\Up(R)}.
 \]
Fixing a tiling $\mathcal{T}$ of $\mathcal{H}(X)$, define 
\[
\wt(X)=\sum_{R\in\TRAT_{\mathcal{T}}(X)}\wt(R).
\]
\end{defn}

\begin{example}
 For example, the TRAT in Figure \ref{TRAT_to_path} has weight $d^3et^3$ since it has respectively three left-arrows and one up-arrow in its 20-tiles, and a total of seven arrows.
\end{example}

One can check combinatorially, for instance following the proof of Theorem
3.1 of \cite{MV15}, that 
\begin{align*}
t\wt(X'20X'')&=\wt(X'2X'')+\wt(X'0X''),\\ 
d\wt(X'21X'')&=\wt(X'1X''),\\ 
e\wt(X'10X'')&=\wt(X'1X'')
\end{align*}
for some 2-TASEP words $X', X''$. This is done by reducing a TRAT of type $X$ to a
tableau of smaller size by removing a north-strip or a
west-strip. We will not reproduce this (fairly standard) proof, and instead we will further build on the connection
between the TRAT and the multiline queues by defining a
weighted version of the multiline queues using the bijection of Section
\ref{sec_bij}, and then proving the recurrences are satisfied on the weighted
MLQs.

\subsection{Multiline queue associated to the inhomogeneous 2-TASEP on a ring}\label{sec_weights}
We introduce a \emph{weighted multiline queue} (WMLQ) that generalizes the definition of the multiline queue of Section \ref{mlq_def}. The ball drop algorithm for the WLMQ is the same as for the usual MLQ, and the type of the WMLQ is also obtained in the same way.

\begin{defn} A 0-ball is \emph{unrestricted} if, immediately following its ball drop, there is an unmarked vacancy to its left with no 1-ball in between. 
\end{defn}

\begin{example}
In Figure \ref{weighted_mlq}, the 0-balls in locations 3 and 5 are unrestricted because at the time they are occupied, the vacancy at location 2 remains unmarked. However, when the 0-balls at locations 6 and 11 are occupied, there are no unmarked vacancies to their left before the nearest 1-ball, so those 0-balls are restricted.
\end{example}

\begin{defn}A \emph{weighted multiline queue} (WMLQ) is a usual multiline queue with weights $t,d,e$ assigned to each entry in the bottom row, as follows: 
\begin{itemize}
\item every marked vacancy receives a weight of $d$,
\item every unrestricted 0-ball receives a weight of $e$,
\item every remaining vacancy or 0-ball receives a weight of $t$.
\end{itemize}
\end{defn}

\begin{defn} The \emph{weight} $\wt(M)$ of a weighted MLQ $M$ is the monomial obtained by taking the product of the weights assigned to the bottom row. In other words, if we define $\urest(M)$ to be the number of unrestricted 0-balls and $\mv(M)$ to be the number of marked vacancies, for $M\in\WMLQ(k,r,\ell)$ we obtain
\[
\wt(M)=d^{\urest(M)}e^{\mv(M)}t^{k+\ell-\urest(M)-\mv(M)}.
\]
\end{defn}

\begin{example} The MLQ in Figure \ref{weighted_mlq} has type $X=12200120200102$ and weight $\wt(M)=d^4e^3t^4$. Observe that the 0-balls in locations 4, 8, and 10 are unrestricted and have weight $e$, and the 0-balls at locations 5, 11, and 13 are restricted and have weight $t$. All vacancies except for the one at location 14 are marked and have weight $d$.

\begin{figure}[!ht]
  \centerline{\includegraphics[height=1in]{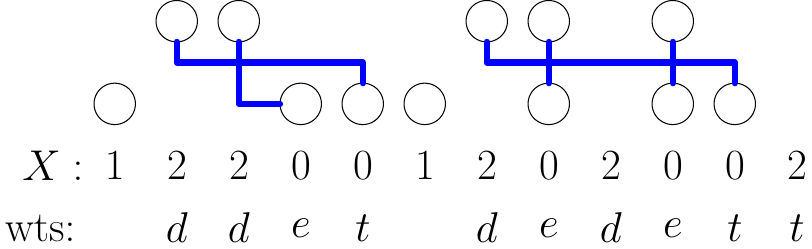}}
\centering
 \caption{An example of $M\in\WMLQ(X)$ for $X=12200120200102$ and weight $\wt(M)=d^4e^3t^4$.}\label{weighted_mlq}
 \end{figure}

\end{example}

\begin{prop}
\[\wt(X) = \sum_{M \in\WMLQ(X)} \wt(M).\]
\end{prop}

\begin{remark}
The definition of the ball drops differs from the usual definition of
bully paths on MLQ's, since the ball drops must occur in order from
right to left, whereas for usual bully paths, the order of the ball
drops is inconsequential. The reason for this in our algorithm is to determine which balls receive weight $e$, and which receive weight $1$. Recall that a bottom row ball receives weight $e$ only if there is an unmarked vacancy to its left immediately following its ball drop. 
\end{remark}

When we set $e=t$, the weight of a weighted MLQ of size $(k,r,\ell)$ reduces to the following: let $\mv(M)$ be the number of marked vacancies of $M$. Then $\wt(M)=d^{\mv(M)}$. From the formula in \cite{AL14}, the weight of $M$ is computed to be $t^{k}\left(\frac{d}{t}\right)^{\mv(M)} = t^{k-\mv(M)}d^{\mv(M)}$, which is equivalent to our own computation up to a factor of $t^{\ell}$.

To show the weighted MLQ's indeed provide a formula for inhomogeneous 2-TASEP probabilities, we give a standard Matrix Ansatz proof.

\begin{lemma}\label{MLQ_recurrences}
Let $M$ be a weighted MLQ whose entries are represented as $M={y_1\ \cdots\ y_{n}\choose x_1\ \cdots\ x_{n}}$ for $x_i\in\{0,1,2\}$ and $y_i\in\{\textbf{v},\textbf{b}\}$, with $\textbf{0},\textbf{1},\textbf{2}$ representing a bottom row 0-ball, 1-ball, or vacancy respectively, and with $\textbf{v},\textbf{b}$ representing a top row vacancy or ball, respectively. Suppose $x_i,x_{i+1}=\textbf{2},\textbf{0}$, and let $M'={y_1\ \cdots\ y_{i-1} \choose x_1\ \cdots\ x_{i-1}}$ and $M''={y_{i+2}\ \cdots\ y_{n}\choose x_{i+2}\ \cdots\ x_{n}}$. Then
\begin{equation}\label{t}
\wt(M)=\begin{cases}
d\wt(M'{y_i \choose \textbf{0}}M'') & \mbox{if } y_{i+1}=\textbf{v}\\
e\wt(M'{y_i \choose \textbf{2}}M'') & \mbox{if } y_{i+1}=\textbf{b}.
\end{cases}
\end{equation}
\end{lemma}

\begin{proof}
If $y_{i+1}=\textbf{v}$, then the 0-ball at $x_{i+1}$ must be occupied by some top row ball that passes location $i$. Thus $x_i$ is necessarily a marked vacancy after $x_{i+1}$ is occupied, and hence has weight $d$. Removing the vacancy and shifting $y_i$ to location $i+1$ has no effect on the rest of the weighted MLQ.

If $y_{i+1}=\textbf{b}$, then removing the entire column at location $i+1$ has no effect on the rest of the MLQ since the ball at $y_{i+1}$ always drops directly on top of the 0-ball at $x_{i+1}$. Moreover, the 0-ball at $x_{i+1}$ acquires weight $e$ since at the time it is occupied, the vacancy at $x_i$ is unmarked - since $y_{i+1}$ is dropped before any top row balls to its left. Thus we obtain Equation \eqref{t}.
\end{proof}

\begin{thm}\label{main_result}
Let $X\in\TASEP(k,r,\ell)$ and let $f(X)=\tr(\DAE(X))$ as defined in Theorem \ref{inhomog_ansatz_thm}. Then
\begin{equation}\label{rec}
\frac{t^{k+\ell}}{d^ke^{\ell}}f(X) = \sum_{M\in \WMLQ(X)} \wt(M).
\end{equation}
\end{thm}

\begin{proof}
Our proof is by induction on the size of $X$. 

For our base case, we consider $X$ which has no instance of $20$. For such $X$, there is a unique MLQ, since no bottom row 0-ball has a vacancy to its left without a 1-ball in between, so every 0-ball must be occupied by a top row ball directly above it. This also implies there are no marked vacancies and every 0-ball is restricted. Thus the $k$ vacancies contribute weight $t^k$ and the $\ell$ 0-balls contribute weight $t^{\ell}$, so $\sum_{M\in \WMLQ(X)} \wt(M)=t^{k+\ell}$. On the Matrix Ansatz side, we directly obtain $f(X)=d^{k}e^{\ell}$. In particular, this is true with $k=0$ or $\ell=0$.

Now, suppose we have $K,L$, such that for any $X\in\TASEP(k,r,\ell)$ with $k \leq K$ and $\ell<L$, Equation \eqref{rec} is satisfied. We will show that Equation \eqref{rec} is also satisfied for $Y\in\TASEP(k,r,\ell)$ where $k,\ell=K,L$. By our base case, if $Y$ has no instance of 20, we are done. Otherwise, let $Y=Y'20Y''$ with $2,0$ in positions $i,i+1$. 

We partition the set of weighted MLQs of type $Y$ into two depending on the contents of column $i+1$:
\begin{multline*}
\WMLQ(Y) = \left\{M\in\WMLQ(Y) : M=M'{y_i\choose \textbf{2}}{\textbf{v}\choose\textbf{0}}M''\right\}\\
 \bigcup  \left\{M\in\WMLQ(Y) : M=M'{y_i\choose \textbf{2}}{\textbf{b}\choose\textbf{0}}M''\right\}.
\end{multline*}
We show the following are bijections:
\begin{equation}\label{bij1}
\WMLQ(Y'0Y'') \Longleftrightarrow \left\{M\in\WMLQ(Y) : M=M'{y_i\choose \textbf{2}}{\textbf{v}\choose\textbf{0}}M''\right\}
\end{equation}
\begin{equation}\label{bij2}
\WMLQ(Y'2Y'') \Longleftrightarrow \left\{M\in\WMLQ(Y) : M=M'{y_i\choose \textbf{2}}{\textbf{b}\choose\textbf{0}}M''\right\}
\end{equation}

For the first equation, let $M=M'{y_i\choose \textbf{2}}{\textbf{v}\choose\textbf{0}}M''\in\WMLQ(Y)$. Then $M'{y_i\choose\textbf{0}}M''\in\WMLQ(Y'0Y'')$ since the 0-ball is still occupied by the same ball in both $M$ and the reduced MLQ. Moreover, given an MLQ $\hat{M}=M'{y_i\choose \textbf{0}}M'' \in\WMLQ(Y'0Y'')$, it is immediate that inserting two vacancies to obtain $M'{y_i\choose \textbf{2}}{\textbf{v}\choose \textbf{0}}M''$ gives back $M$, thus establishing the bijection.

Similarly, let $M=M'{y_i\choose \textbf{2}}{\textbf{b}\choose\textbf{0}}M''\in\WMLQ(Y)$. Then $M'{y_i\choose\textbf{2}}M''\in\WMLQ(Y'2Y'')$ since the column at location $i+1$ in $M$ had no effect on the rest of the weighted MLQ, keeping its type the same minus the 0 in location $i+1$. It's clear this is a bijection as well. Thus we obtain

\[
\sum_{M\in\WMLQ(Y)} \wt(M)=\sum_{\substack{M',\ M'',\ x\\\mbox{s.t.}\ M\in\WMLQ(Y)}} \wt(M' {x\choose \textbf{2}}{\textbf{v}\choose \textbf{0}}M'') + \wt(M' {x\choose \textbf{2}}{\textbf{b}\choose \textbf{0}}M'').
\]
By Lemma \ref{MLQ_recurrences}, this equals
\[
\sum_{\substack{M',\ M'',\ x\\\mbox{s.t.}\ M\in\WMLQ(Y)}} d\wt(M' {x\choose \textbf{0}}M'') + e\wt(M' {x\choose \textbf{2}}M''),
\]
which reduces to 
\[
\sum_{M\in\WMLQ(Y'20Y'')}\wt(M)= \sum_{M\in\WMLQ(Y'0Y'')}d\wt(M) + \sum_{M\in\WMLQ(Y'2Y'')}e\wt(M)
\]
by our arguments above. Consequently, by our induction assumption and Theorem \ref{inhomog_ansatz_thm}, this equals
\[
\sum_{M\in\WMLQ(Y)}\wt(M)=\frac{1}{t^{k+\ell-1}}\left(d \cdot d^{k-1}e^{\ell}f(Y'0Y'')+e \cdot d^ke^{\ell-1} f(Y'2Y'')\right) = \left(\frac{1}{t}\right)\frac{d^ke^\ell}{t^{k+\ell-1}}f(Y),
\]
thus completing the proof.
\end{proof}

\begin{remark}
We obtain a solution to the inhomogeneous TASEP studied in \cite{FM07} by setting $e=t=x_1$, $e=x_2$, where $x_1$ is the rate of the transition $20\rightarrow 02$, $10\rightarrow 01$, and $x_2$ is the rate of the transition $21 \rightarrow 12$. 
\end{remark}

\begin{remark}
Observe that the parameter $t$ is unnecessary, since for any $M\in\WMLQ(k,r,\ell)$ and $R\in\TRAT(k,r,\ell)$, the monomials $\wt(M)$ and $\wt(R)$ have degree $k+\ell$. We can thus simplify all expressions by setting $t=1$ without losing any information, and we will do so for the remainder of the paper.
\end{remark}

\begin{thm}\label{TRAT_MLQ_thm}
 Let $T=\trat(M)$ for weighted MLQ $M$. Then 
 \[
 \wt(T)=\wt(M). 
 \]
\end{thm}

\begin{proof}
Every north-strip in a TRAT has exactly one up-arrow, and every west-strip has exactly one left-arrow. If an up-arrow (resp.~left-arrow) is not in a 20-tile, it is in a 10-tile (resp.~21-tile). Following the MLQ-TRAT bijection in Section \ref{sec_bij}, by construction the left-arrows in the 20-tiles are precisely those that correspond to marked vacancies in M, and hence $\Left(T)=\mv(M)$. Marked vacancies contribute $d$ to $\wt(M)$, so the power of $d$ is the same for $\wt(T)$ and $\wt(M)$.

Recall that we call a free tile in a TRAT one that is not pointed to by (or already contains) a left-arrow. In each north-strip, the up-arrow is placed in the bottom-most free tile. Let $M$ have a 0-ball at location $j$ with dropping weight $w$. When this 0-ball is unrestricted, there is an unmarked gap to its left at the time it is occupied. Balls are dropped from right to left and north-strips of $T$ are filled from right to left: thus an unmarked gap to the left of a 0-ball implies there is a free 20-tile in north-strip $j$ above the $w$ 20-tiles containing the left-arrows. Consequently, the up-arrow is contained in a 20-tile in strip $j$, contributing a weight of $e$. On the other hand if the 0-ball is restricted, the opposite occurs, and there are no free 20-tiles in north-strip $j$. Thus the up-arrow is contained in a 10-tile, and so $\Up(T)=\urest(M)$ with the latter contributing a weight of $e$ to $\wt(M)$. Thus the power of $e$ is the same for $\wt(T)$ and $\wt(M)$, from which we can deduce that $\wt(T)=\wt(M)$.
\end{proof}

\begin{figure}[!ht]
  \centerline{\includegraphics[width=\linewidth]{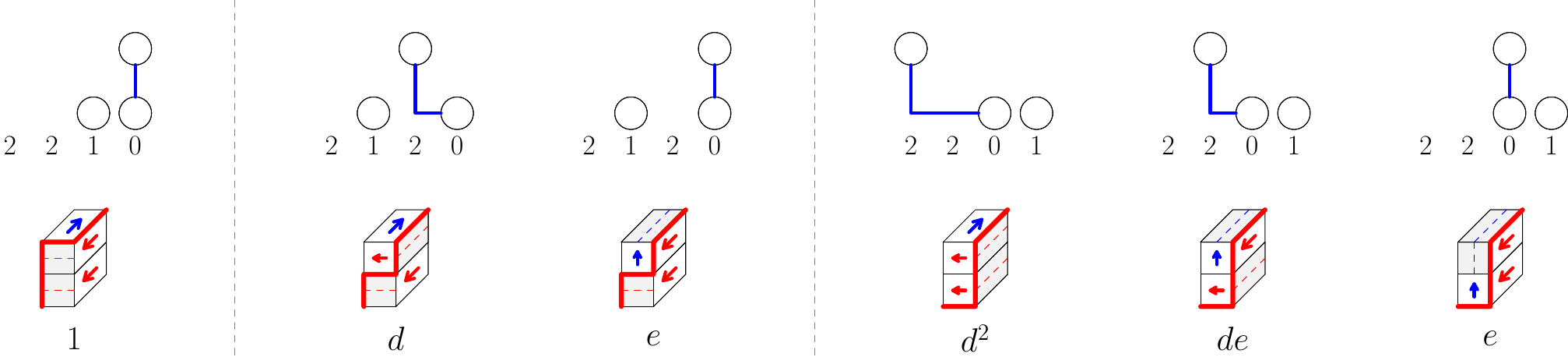}}
\centering
 \caption{All six elements of $\WMLQ(2,1,0)$ and the corresponding elements of $\TRAT(2,1,0)$ along with their respective weights (note that $t$ is omitted, since we have set it to equal 1).}\label{DDEA}
 \end{figure}

\begin{example}
In Figure \ref{DDEA} we show all multiline queues in the set $\WMLQ(2,1,1)$ and their corresponding TRAT, along with the weights. From Theorem \ref{main_result} we conclude that
\begin{align*}
\Pr(2210)&=\frac{1}{\mathcal{Z}_{2,1,1}}1\\
\Pr(2021)&=\frac{1}{\mathcal{Z}_{2,1,1}}(d+e)\\
\Pr(2201)&=\frac{1}{\mathcal{Z}_{2,1,1}}(d^2+de+e),
\end{align*}
where $\mathcal{Z}_{2,1,1}=1+d+2e+d^2+de$.
\end{example}

\begin{cor}
The TRAT Markov chain projects to the inhomogeneous TASEP when the stationary probability of a TRAT is its weight. Thus for $X$ a state of the inhomogeneous TASEP and $\mathcal{T}$ some fixed tiling of $\mathcal{H}(X)$, 
\[
 \Pr(X) \propto \sum_{T\in \TRAT_{\mathcal{T}}(X)}\wt(T).
\]
\end{cor}

\section{The 2-TASEP with open boundaries and acyclic multiline queues}\label{sec_TASEP_open}

A nice consequence of our TRAT-MLQ bijection is that we can apply the same methods to obtain analogous results for the two-species totally asymmetric simple exclusion process (2-TASEP) with open boundaries. The 2-TASEP is a Markov chain whose states are configurations of particles of type 0, 1, and 2 on a finite lattice with open boundaries. The states are represented by words $X=X_1\ldots X_n$ with $X_i \in \{2,1,0\}$ and the possible transitions are:
\begin{itemize}
\item two adjacent particles $X_iX_{i+1}$ can swap with rate 1 if $X_i>X_{i+1}$,
\item at $X_1$, particle 0 can be replaced by particle 2 with rate $\alpha$, and
\item at $X_n$, particle 2 can be replaced by particle 0 with rate $\beta$,
\end{itemize}
where $0 \leq \alpha, \beta \leq 1$ are parameters dictating the rates of \emph{entry} and \emph{exit} of particles at the boundaries of the lattice. The number of type 1 particles is conserved. Thus we define $\TASEP(n,r)$ to be the set of 2-TASEP words of length $n$ with exactly $r$ particles of type 1.

\begin{remark}
Classically, the 2-TASEP is described as a model describing the dynamics of two species of particles, \emph{heavy} and \emph{light}, hopping left and right on a a lattice of $n$ sites, such that heavy particles can replace a vacancy at the first location, and can be replaced by a vacancy at the $n$'th location. In the bulk, any particle can swap places with a vacancy or the heavy particle can swap with an adjacent light particle on its right. In our case, the heavy particles, light particles and vacancies are represented by particles of type 2, 1, and 0 respectively. 
\end{remark}

The 2-TASEP has been studied by many including \cite{Uch08, Ari06, ALS09, DS05}. A Matrix Ansatz due to Uchiyama expresses the stationary probabilities of the 2-TASEP as a matrix product, as follows.

\begin{thm}[\cite{Uch08}]\label{uchi}
Let $D, A, E$ be matrices and $\langle w|, |v\rangle$ vectors satisfying:
\begin{align*}\label{2-TASEP_open_ansatz}
DE&=D+E&\qquad\langle w|E &= \frac{1}{\alpha}\langle w|\\
DA&=A&\qquad D|v\rangle &= \frac{1}{\beta}|v\rangle\\
AE&=A&&
\end{align*}
Let $X$ be a state of the 2-TASEP of size $(k,r,\ell)$ with open boundaries. Then the stationary probability is given by: 
\[
\Prob(X) =\frac{1}{Z_{n,r}} \alpha^k\beta^{\ell} \langle w|\DAE(X)|v\rangle,
\]
where $n=k+r+\ell$ and the partition function is $Z_{n,r}=[y^r]\langle w|(D+yA+E)^n|v\rangle$ where $[y^r]p(y)$ denotes the coefficient of $y^r$ in $p(y)$.
\end{thm}

A set of matrices that satisfy the conditions of the Ansatz are $D=(D_{ij})_{i,j}$, $A=(A_{ij})_{i,j}$, and $E=(E_{ij})_{i,j}$ such that:
\[
D_{ij}=\begin{cases} \alpha & j=i+1\\ 0& \mbox{otherwise}, \end{cases}
\qquad
A_{ij}=\begin{cases} \beta^i & j=0\\ 0 & \mbox{otherwise}, \end{cases}
\qquad
E_{ij}=\begin{cases} \beta^i & j=0\\ \alpha \beta^{i-j+1} & \mbox{otherwise}, \end{cases}
\]
with $\langle w| = (1,0,0,\ldots)$ and $v|\rangle = (1,1,1,\ldots)^T$.

\begin{example} For example, 
\[
\Prob(20201210)=\frac{1}{Z_{8,2}} \alpha^3\beta^{3} \langle w|DEDEADAE|v\rangle=\frac{1}{Z_{8,2}} \alpha^3\beta^3(2\alpha^3\beta^3+2\alpha^2\beta^3+\alpha\beta^3),
\] 
where $D, A, E$ and $\langle w|, |v\rangle$ are any matrices and vectors satisfying Equation \ref{uchi}, for instance those given above.
\end{example}

A tableaux solution for the stationary probabilities of the 2-TASEP with open boundaries was discovered by the author in \cite{Man16}, and shortly thereafter generalized in a joint work with Viennot in \cite{MV15} by introducing the \emph{rhombic alternative tableaux} (RAT), on which the TRAT are based. In this section we define a specialization of the RAT which corresponds to the 2-TASEP; for the general definition of a RAT, see \cite{MV15}.

A RAT of type $X$ is obtained by taking the tiled northwest portion of $\mathcal{H}(X)$ (with no restriction on $X$), and filling that region with up-arrows and left-arrows. We denote the region of $\mathcal{H}(X)$ northwest of $P(X)$ by $\Gamma(X)$, which we call a \emph{rhombic diagram}. We carry over the definition of a tiled rhombic diagram, north-strips, west-strips, up-arrows, and left-arrows from previous sections. Recall that when we say a tile is \emph{pointed at} by an up-arrow (resp.~left-arrow), that means there is an up-arrow below the tile in the same north-strip (resp.~left-arrow to the right of the tile in the same west-strip). 

\begin{defn}\label{RAT_def}
A \emph{rhombic alternative tableau} (RAT) of type $X\in\TASEP(n,r)$ is a rhombic diagram $\Gamma(X)$ with some tiling $\mathcal{T}_X$ that is filled with up-arrows and left-arrows according to the following filling rules:
\begin{itemize}
\item a tile must be empty if it is pointed at by an up-arrow or a left-arrow, and
\item if a tile is not pointed at by an arrow, it must contain an up-arrow or a left-arrow.
\end{itemize}
\end{defn}

\begin{defn}
The \emph{weight} of a RAT $R$ of size $(k,r,\ell)$ is given by
\[
\wt(R) = \alpha^{k+\#\{\mbox{up-arrows}\}}\beta^{\ell+\#\{\mbox{left-arrows}\}}.
\]
\end{defn}

For $X\in\TASEP(n,r)$, we define $RAT_{\mathcal{T}}(X)$ to be the set of RAT of type $X$ with some fixed tiling $\mathcal{T}$. We denote the set of all RAT of size $(n,r)$ by $RAT(n,r)$. More precisely,
\[
RAT(n,r)=\bigcup_{Y\in\TASEP(n,r)} RAT_{\mathcal{T}_Y}(Y)
\]
where $\{\mathcal{T}_Y\}_Y$ is some set of tilings of rhombic diagrams $\{\Gamma(Y)\}_Y$ where $Y$ ranges over all possible states of the 2-TASEP. 

The following result is Theorem 3.1 in \cite{MV15}, and is proved with the canonical Matrix Ansatz technique.

\begin{thm}[\cite{MV15}]\label{RAT_thm}
The steady state probability of state $X\in\ASEP(n,r)$ is
\[
\Prob(X)=\frac{1}{\mathcal{Z}_{n,r}}\sum_{R\in RAT_{\mathcal{T}_X}(X)}\wt(R),
\]
where $\mathcal{T}_X$ is a fixed tiling of $\Gamma(X)$ and $\mathcal{Z}_{n,r}=\sum_{R\in RAT(n,r)}\wt(R)$ is the partition function.
\end{thm}

In Section \ref{sec_amlq}, we introduce a new object, the \emph{acyclic multline queue} (AMLQ), which is derived from the usual multiline queue, and is in bijection with the RAT. In Section \ref{2-TASEP_open_inhomog} we generalize the 2-TASEP to an inhomogeneous process similar to the inhomogeneous TASEP on a ring of Section \ref{sec_inhomog}, and likewise generalize the RAT and the AMLQ to solve the inhomogeneous model.

\subsection{Acyclic multiline queues}\label{sec_amlq}

The connection between rhombic tableaux and multiline queues naturally extends to the open boundary case of the 2-TASEP. Following the idea of the TRAT-MLQ bijection of Section \ref{sec_bij}, we define \emph{acyclic multiline queues}, which we also call AMLQs.

\begin{defn}
 An \emph{acyclic MLQ} (AMLQ) of type $X\in\TASEP(n,r)$ is a configuration of two rows of balls on a lattice of size $2\times n$ with open boundaries. There are $\ell\leq n-r$ balls in the top row and $\ell+r$ balls in the bottom row, with the following restriction: for each $1\leq i\leq \ell+r$, the $i$'th bottom row ball (from the left) has at least $i$ top row balls weakly to its left. We denote the set of AMLQs of type $X$ by $\AMLQ(X)$, and we denote by $\AMLQ(n,r)$ the set of acyclic MLQs of size $(n,r)$: 
 \[\AMLQ(n,r)=\bigcup_{X\in\TASEP(n,r)}\AMLQ(X).\]
\end{defn}

In other words, an acyclic MLQ is an MLQ with open boundaries (which is not invariant under cyclic shifts), and in which every top row ball occupies a bottom row ball to its right. Figure \ref{not_amlq} shows an example of AMLQs, including a configuration which is \emph{not} an AMLQ.

\begin{figure}[!ht]
  \centerline{\includegraphics[width=\linewidth]{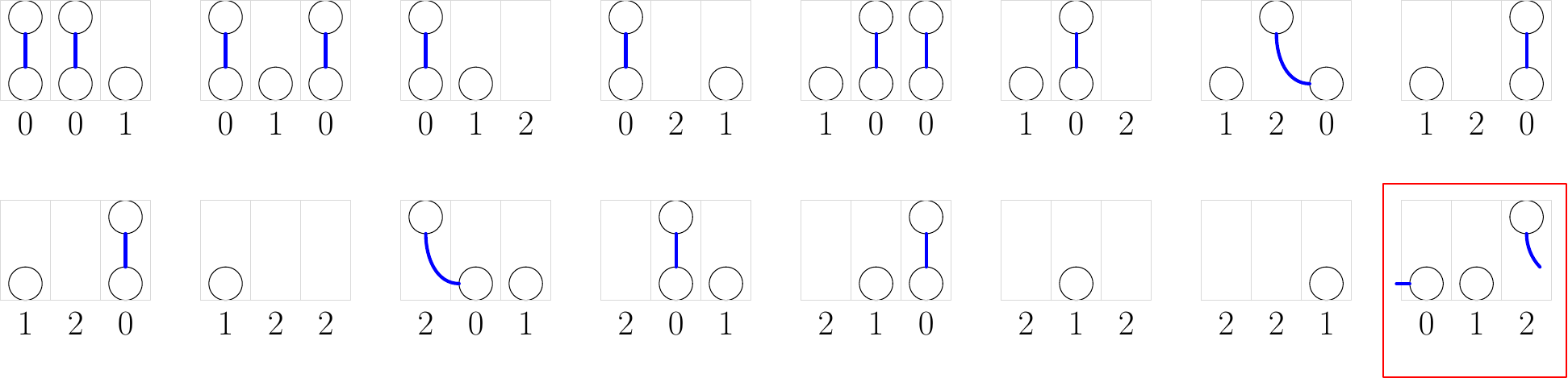}}
\centering
 \caption{All the acyclic MLQs of size $(3,1)$ are shown. The rightmost (boxed) configuration is not an AMLQ since the top row ball must wrap around to occupy the bottom row.}\label{not_amlq}
 \end{figure}

\begin{lemma}
$\AMLQ(X)$ is in bijection with $\MLQ(1X1)$.
\end{lemma}

\begin{proof}
Let $A\in\AMLQ(X)$. Let $A_M$ be an MLQ obtained by appending a column to the left and right of $A$, containing a vacancy in the top row and a 1-ball in the bottom row. The leftmost column of $A_M$ trivially contains a 1-ball, and since every top row ball in $A$ occupies some ball weakly to its right without wrapping around, the bottom row ball at the rightmost location of the $A_M$ must remain unoccupied; thus $A_M\in\MLQ(1X1)$. On the other hand, let $A_M\in\MLQ(1X1)$ be cyclically shifted so that its type read from left to right is $1X1$. The right-most 1-ball must remain unoccupied, so all top row balls must occupy bottom row balls without wrapping around. By chopping off the leftmost and rightmost columns of $A_M$, we get back $A\in\AMLQ(X)$.    
\end{proof}

We fix some definitions to simplify notation.
 \begin{defn}
 Let $A$ be an AMLQ. 
 \begin{itemize} 
 \item $\urest(A)$ is the number of unrestricted 0-balls in $A$.
 \item $\mv(A)$ is the number of marked vacancies in $A$.
\item $\ufree(A)$ is be the number of restricted 0-balls to the left of the leftmost 1-ball in $A$. 
\item $\lfree(A)$ is be the number of unmarked vacancies to the right of the rightmost 1-ball in $A$.
 \end{itemize}
\end{defn}

\begin{defn}\label{wt_AMLQ}
 The \emph{weight} of an acyclic MLQ $A\in AMLQ(n,r)$ is
 \[
  \wt(A) = \alpha^{n-r-\ufree(A)}\beta^{n-r-\lfree(M)}.
 \]
\end{defn}

\begin{thm}\label{amlq_main}
 Let $X$ be a state of the two-species TASEP of size $(n,r)$. Then 
 \[
  \Pr(X) = \frac{1}{\mathcal{Z}_{n,r}}\sum_{A \in \AMLQ(X)} \wt(A),
 \]
 where $\mathcal{Z}_{n,r} = \sum_{A\in \AMLQ(n,r)} \wt(A)$ is the partition function.
\end{thm}

In Figure \ref{not_amlq}, the $\alpha,\beta$ weights of all AMLQs of size $(3,1)$ are given (after setting the other variables to equal 1).

The proof of our theorem is through a weight-preserving bijection of acyclic MLQs with RAT. Let us denote the set of fillings of $\Gamma(X)$ with tiling $\mathcal{T}$ by $\RAT_{\mathcal{T}}(X)$. 

\begin{figure}[!ht]
  \centerline{\includegraphics[width=\linewidth]{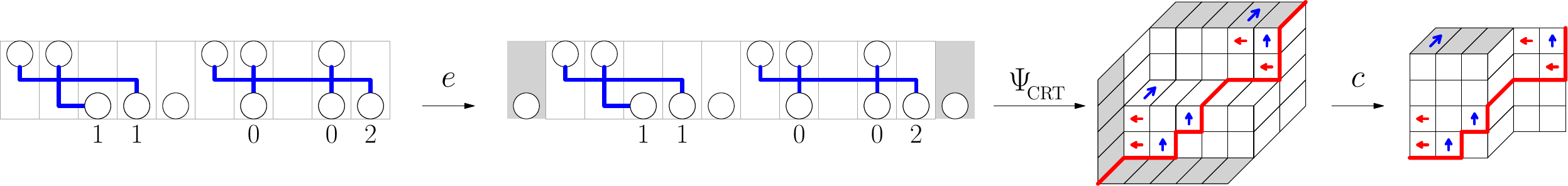}}
\centering
 \caption{Let $X=220012020010202$. From left to right, we have: $A\in\AMLQ(X) \longleftrightarrow A_M\in\MLQ(1X1) \longleftrightarrow \trat(A_M)\in\TRAT_{\mathcal{T}_X}(1X1) \longleftrightarrow \rat(A)\in\RAT(X)$. The highlighted columns in $A_M$ which are appended to $A$ correspond to the highlighted diagonal strips in $\trat(A_M)$ which are subsequently removed to obtain $\rat(A)$.}\label{amlq_bij}
 \end{figure}

\begin{proof}
Let $A\in\AMLQ(X)$ for $X\in\TASEP(n,r)$, and let $X'=1X1\in \TASEP(n+2,r+2)$. Let $A_M\in\MLQ(1X1)$ be the MLQ obtained by appending a column containing a 1-ball in the bottom row to the left and right of $A$. Let $T=\trat(A_M)$ be the TRAT obtained by applying the ball drop algorithm to $A_M$. Now apply flips to the tiling of $T$ until the rightmost (resp.~leftmost) diagonal strip consists of a row of $\ell$ adjacent 10-tiles (resp.~$k$ adjacent 21-tiles), obtaining tiling $\mathcal{T}$ in which the leftmost and rightmost diagonal strips are aligned with the north and west boundaries of $T$.

We define $T_U$ to be the rhombic tableau obtained by taking the region of $T$ northwest of $P(X)$. $T_U$ satisfies the rules of Definition \ref{RAT_def}, and so $T_U\in\RAT_\mathcal{T}(1X1)$. We claim that the region of $T$ southeast of the path $P(X)$ contains no arrows. Suppose a north-strip of $T$ has its top-most tile, which is a 10-tile contained in the rightmost diagonal strip, free. Then that 10-tile will necessarily contain an up-arrow. Similarly, suppose a west-strip of $T$ has its left-most tile, which is a 21-tile contained in the leftmost diagonal strip, free. Then that 21-tile will necessarily contain a left-arrow. Thus every arrow in $T$ is contained in $T_U$, and so a TRAT $T\in \TRAT_{\mathcal{T}_X}(1X1)$ can be uniquely recreated from a rhombic tableau $T_U\in \RAT(1X1)$. Consequently, this map is a bijection.

Now define $R$ to be the tableau obtained by chopping off the rightmost and leftmost diagonal strips of $T_U$: recall that with the tiling $\mathcal{T}$, these strips are bordering the northwest boundary of $T_U$, so chopping them off results in a proper rhombic diagram of type $X$ with a filling with up-arrows and left-arrows that still satisfy Definition \ref{RAT_def}; thus $R\in\RAT_{\mathcal{T}}(X)$. Moreover, $T_U$ can be recreated from $R$ by re-attaching the external diagonal strips and placing a up-arrow in the topmost tile of every north-strip that is free of an up-arrow, and a left-arrow in the leftmost tile of every west-strip that is free of a left-arrow. We call $\rat(A)=R$. Thus 
\[
\rat: \AMLQ(X) \rightarrow \RAT_\mathcal{T}(X)
\] 
is a bijection. 

Define $\ufree(R)$ to be the number of north-strips that are free of up-arrows and $\lfree(R)$ to be the number of west-strips that are free of left-arrows. Then 
\begin{align*}
\wt(R)&=\alpha^k\beta^{\ell}\alpha^{\#\ \mbox{up-arrows}}\beta^{\#\ \mbox{left-arrows}}\\
&=\alpha^{n-r-\ufree(R)}\beta^{n-r-\lfree(R)}.
\end{align*}

By our construction, $\ufree(R)$ (resp.~$\lfree(R)$) is the number of tiles containing an up-arrow in the rightmost diagonal strip (resp.~left-arrow in the leftmost diagonal strip) of $R$. By following the ball drop algorithm, we see that an up-arrow in the rightmost diagonal strip precisely corresponds to the unmarked vacancies left of the leftmost 1-ball in $A$, and a left-arrow in the leftmost diagonal strip precisely corresponds to the restricted 0-balls right of the rightmost 1-ball in $A$. Thus $\ufree(R)=\ufree(A)$ and $\lfree(R)=\lfree(A)$ from Definition \ref{wt_AMLQ}.

For $X\in\TASEP(n,r)$, let $A\in\AMLQ(X)$ and $R=\rat(A)\in\RAT(X)$. By the above $\wt(R)=\wt(A)$, and so with Theorem \ref{RAT_thm}, we obtain
\[\Pr(X)=\frac{1}{\mathcal{Z}_{n,r}}\sum_{R\in\RAT_{\mathcal{T}}(X)} \wt(R) = \frac{1}{\mathcal{Z}_{n,r}}\sum_{A\in\AMLQ(X)} \wt(A),\]
as desired, where $\mathcal{Z}_{n,r}=\sum_{A\in\AMLQ(n,r)}\wt(A)$.
\end{proof}

\subsection{Combinatorics of the inhomogeneous 2-TASEP with open boundaries}\label{2-TASEP_open_inhomog}

In the inhomogeneous 2-ASEP, swaps of different species of particles occur at different rates which depend on both particles involved in the swap, analogous to the inhomogeneous generalization in Section \ref{sec_inhomog}. Let $X\in\TASEP(n,r)$ represent a state of the 2-ASEP. The transitions on the inhomogeneous 2-ASEP Markov chain are:
\begin{align*}
X'20X''&  \mathrel{\mathop{\rightarrow}^{\mathrm{t}}}  X'02X''&\qquad 0X' & \mathrel{\mathop{\rightarrow}^{\mathrm{\alpha}}} 2X'\\
X'21X''& \mathrel{\mathop{\rightarrow}^{\mathrm{d}}}  X'12X''&\qquad X'2 & \mathrel{\mathop{\rightarrow}^{\mathrm{\beta}}} X'0\\
X'10X''&  \mathrel{\mathop{\rightarrow}^{\mathrm{e}}}  X'01X''&&
\end{align*}
where $0 \leq t, d, e, \alpha, \beta \leq 1$ are parameters describing the hopping rates. (When $t=d=e=1$, we recover the usual 2-TASEP.)

The following Matrix Ansatz is the inhomogeneous modification of the canonical Derrida-Evans-Hakim-Pasquier Matrix Ansatz \cite{DEHP} and the two-species Matrix Ansatz, studied by \cite{Ari06, Uch08}.

\begin{thm}\label{inhomog_ASEP_ansatz_thm}
Let $D, A, E$ be matrices and $\langle w|, |v\rangle$ vectors satisfying:
\begin{align}\label{inhomog_ASEP_ansatz}
tDE&=D+E\qquad & \langle w|E &= \frac{1}{\alpha}\langle w|\\ \nonumber
dDA&=A\qquad & D|v\rangle &= \frac{1}{\beta}|v\rangle\\ \nonumber
eAE&=A
\end{align}
then the stationary probability of state $X$ of the inhomogeneous 2-TASEP with open boundaries of size $(n,r)$ is given by

\[
\frac{1}{Z_{n,r}}\langle w|\DAE(X)|v\rangle,
\]
where $Z_{n,r}=[y^r]\langle w| (D+yA+E)^n |v\rangle$.
\end{thm}

We can define weighted RAT and weighted acyclic MLQs in the same way that we define weighted TRAT and weighted MLQs, respectively. As before, we fix $t=1$ by normalizing over all parameters. 

Recall that for a TRAT $R$, we denote by $\Up(R)$ and $\Left(R)$ the number of 20-tiles containing an up-arrow and a left-arrow, respectively. We carry over this definition for the RAT. Also recall that $\ufree$ is the number of north-strips not containing an up-arrow, and $\lfree$ is the number of west-strips not containing a west-arrow.

\begin{defn}\label{trat_weight_inhomog}
Let $\wt(R)$ be the $\alpha,\beta$ weight of a RAT $R$. The \emph{enhanced weight} $\wt_e(R)$ is defined as follows:
\begin{align*}
\wt_e(R) &= \wt(R)d^{\Left(R)+\lfree-k}e^{\Up(R)+\ufree-\ell}\\
&=\alpha^{n-r-\ufree(R)}\beta^{n-r-\lfree(R)}d^{\Left(R)+\lfree-k}e^{\Up(R)+\ufree-\ell}.
\end{align*}
In other words, the power of $d$ is $\Left(R)$ minus the total number of left-arrows, which is the same as $d^{-1}$ to the power of the number of \emph{left-arrows contained in 21-tiles}. Similarly, the power of $e$ is $\Up(R)$ minus the total number of up-arrows, which is the same as $e^{-1}$ to the power of the number of \emph{up-arrows contained in 10-tiles}.
\end{defn}

For example, in Figure \ref{amlq_bij}, we obtain $\wt(R)=\alpha^{12}\beta^{12}$ since $\lfree(R)=\ufree(R)=1$, and $\wt_e(R)=\wt(R)d^5e^2$ since $\Left(R)=5$ and $\Up(R)=2$.

Equivalence of tilings is required for the weighted RAT to be well-defined. For this we invoke Lemma \ref{flip_lemma}. Note that every filling of a hexagonal configuration of three tiles contributes weight 1 since such configurations can never have 20-tiles that contain arrows (see Figure \ref{flip_bij}). Thus $\sum_{R\in \RAT_{\mathcal{T}}(X)} \wt_e(R)=\sum_{R\in \RAT_{\mathcal{T}'}(X)} \wt_e(R)$ for any two tilings $\mathcal{T}$ and $\mathcal{T}'$, permitting the following definition.

\begin{defn}
 We define the weight of a state $X\in\ASEP(n,r)$ to be the weight generating function of all RAT of type $X$ with some fixed tiling $\mathcal{T}$:
 \[
 \weight(X)=\sum_{R\in \RAT_{\mathcal{T}}(X)} \wt_e(R).
 \]
 \end{defn}

\begin{thm}\label{RAT_weighted_thm}
 Let $X\in\ASEP(n,r)$ be a state of the inhomogeneous two-species TASEP with open boundaries and parameters $\alpha$, $\beta$, $t=1$, $d$, $e$ governing the transition rates. The stationary probability of state $X$ is
 \[
  \Prob(X) = \frac{1}{\mathcal{Z}_{n,r}} \sum_{R \in \RAT_{\mathcal{T}}(X)} \wt_e(R).
 \]
 for some fixed tiling $\mathcal{T}$, and where $\mathcal{Z}_{n,r}=\sum_{R\in\RAT(n,r)}$ is the partition function.
 \end{thm}

To prove this result, we recall the Matrix Ansatz proof for the RAT in \cite{MV15}, except that when we apply the recurrence relation given by the Matrix Ansatz at the 2, 0 corner corresponding to $X=X'20X''$, we include the $d$ and $e$ weights on the contents of that 20-corner tile. We obtain relations identical to those obtained in the proof of Theorem \ref{main_result} for the weighted MLQs. We leave it to the reader to fill in the details.

Following our bijection with RAT, we obtain an enhanced weight $\wt_e$ for the weighted AMLQs which is a monomial in $\alpha,\beta,d,e$. The fact that our bijection between $\RAT(n,r)$ and $\AMLQ(n,r)$ is weight preserving follows immediately from the fact that the TRAT-MLQ bijection is weight preserving. Consequently, we obtain a formula for stationary probabilities of the inhomogeneous 2-TASEP with open boundaries in terms of AMLQs.

\begin{defn}
 Let $A\in\AMLQ(n,r)$, and let $\wt(A)$ be the $\alpha,\beta$ weight of $A$ from Definition \ref{wt_AMLQ}. Define the enhanced weight of $A$ to be
 \begin{align*}
  \wt_e(A)&=\wt(A) d^{\mv(A)+\lfree(A)-k} e^{\urest(A)+\ufree(A)-\ell} \\
  &=\alpha^{n-r-\ufree(A)}\beta^{n-r-\lfree(A)}d^{\mv(A)+\lfree(A)-k}e^{\urest(A)+\ufree(A)-\ell} .
 \end{align*}
\end{defn}

\begin{example}
In Figure \ref{amlq_bij}, we have $A\in\AMLQ(X)$ on the left and $\rat(A)\in\RAT(X)$ on the right for $X=220012020010202\in\ASEP(15,2)$. For both, $\ufree(A)=\ufree(\rat(A))=1$ and $\lfree(A)=\lfree(\rat(A))=1$. Moreover, $\mv(A)=\Left(\rat(A))=5$ and $\urest(A)=\Up(\rat(A))=2$, and so both objects have $\wt(A)=\wt(\rat(A))=\alpha^{13-1}\beta^{13-1} = (\alpha\beta)^{12}$ and enhanced weight $\wt_e(A)=\wt_e( \rat(A))=(\alpha\beta)^{12}d^5e^2$.

For another example, see Figure \ref{not_amlq}, in which the weights of all elements of $\AMLQ(3,1)$ are given.
\end{example}

 \begin{cor}\label{AMLQ_weighted_thm}
 Let $X\in\ASEP(n,r)$ be a state of the inhomogeneous TASEP with open boundaries and parameters $\alpha$, $\beta$, $t=1$, $d$, $e$ governing the rates of transitions. The stationary probability of state $X$ is
 \[
 \Prob(X)=\frac{1}{\mathcal{Z}_{n,r}} \sum_{A \in\AMLQ(X)} \wt_e(A),
 \]
 where $\mathcal{Z}_{n,r}=\sum_{A\in\AMLQ(n,r)}\wt_e(A)$ is the partition function.
 \end{cor}

 \section{Markov chains that project to the 2-TASEP}\label{sec_markov}

The structure of the toric rhombic tableaux yields a natural Markov chain on these tableaux that projects to the two-species TASEP on a ring, in the flavor of the Markov chain on the rhombic alternative tableaux projecting to the two-species ASEP in \cite{Man15a}, which in turn generalized the Markov chain on the alternative tableaux in \cite{CW07a}. As a result, we also obtain a Markov chain on the two-species multiline queues by following the bijection of Section \ref{sec_bij}, which we describe in Section \ref{sec_MLQ_markov}. These Markov chains extend naturally to the inhomogeneous TRAT and weighted multiline queues.

By projection of Markov chains, we mean the following definition, which is precisely Definition 3.20 from \cite{CW07a}. 

\begin{defn}
Let $M$ and $N$ be Markov chains on finite sets $X$ and $Y$, and let $f$ be a surjective map from $X$ to $Y$. We say that $M$ \emph{projects} to $N$ if the following properties hold:
\begin{itemize}
\item If $x_1, x_2 \in X$ with $Prob_M(x_1 \rightarrow x_2) > 0$, then $Prob_M(x_1 \rightarrow x_2) = Prob_N(f(x_1) \rightarrow f(x_2))$.
\item If $y_1$ and $y_2$ are in $Y$ and $Prob_N(y_1 \rightarrow y_2)>0$, then for each $x_1 \in X$ such that $f(x_1) = y_1$, there is a unique $x_2 \in X$ such that $f(x_2) = y_2$ and $Prob_M (x_1 \rightarrow x_2) > 0$; moreover, $Prob_M (x_1 \rightarrow x_2) = Prob_N (y_1 \rightarrow y_2)$.
\end{itemize}
\end{defn}

Furthermore, we have the following Proposition \ref{mc_walk}, which implies Corollary \ref{cor_projection} below.

Let $\Prob_M(x_0 \rightarrow x; t)$ denote the probability that if we start at state $x_0$ at time 0, then we are in state $x$ at time $t$. From the following proposition of \cite{CW07a}, we obtain that if $M$ projects to $N$, then a walk on the state diagram of $M$ is indistinguishable from a walk on the state diagram of $N$.

\begin{prop}\label{mc_walk}
Suppose that $M$ projects to $N$. Let $x_0 \in X$ and $y_0,y_1 \in Y$ such that $f(x_0)=y_0$. Then 
\[ 
\Prob_N(y_0 \rightarrow y_1) = \sum_{x' \mbox{ s.t. } f(x')=y_1} \Prob_M(x_0 \rightarrow x_1)
\]
\end{prop}

Let $\Omega^{\TRAT}_{MC}$ be the Markov chain on the TRAT. We call the Markov chain on the 2-TASEP on a ring the \emph{TASEP chain}. 

\begin{cor}\label{cor_projection}
Suppose the TRAT chain $\Omega^{\TRAT}_{MC}$ projects to the TASEP chain. Let $X$ be a state of the 2-TASEP. Then the steady state probability of state $X$ in the TASEP chain is equal to the sum of the steady state probabilities that $\Omega^{\TRAT}_{MC}$ is in any of the states $R \in \TRAT_{\mathcal{T}_X}(X)$.
\end{cor}

\subsection{Definition of the TRAT Markov chain $\Omega^{\TRAT}_{MC}$}

In this section, we no longer require fixing a particular tiling on $\mathcal{H}(X)$. 

\begin{defn}
A \emph{corner} of a tableau is a consecutive pair of 2 and 0, 2 and 1, or 1 and 0 edges (we call these 20-, 21-, and 10-corners, respectively). In particular, when a tableau has type $X=1Y2$, then the 1 and 2 edges at the opposite ends of $P(X)$ also form a 21-corner, since the tableau is a torus. Similarly, an \emph{inner corner} is a consecutive pair of 0 and 2, 0 and 1, or 1 and 2 edges (we call these 02-, 01-, and 12-corners, respectively).
\end{defn}

Let $R$ be a tableau with $a$ 20-corners, $b$ 21-corners, and $c$ 10-corners. Then there are $a+b+c$ possible transitions out of $R$. The Markov transitions are based on the following: each corner corresponds, up to tiling equivalence, to either a north-strip with an up-arrow in its bottom-most box (the \emph{up-arrow case}), or a west-strip with a left-arrow in its right-most box (the \emph{left-arrow case}). We define insertion of strips precisely below.

\begin{defn}
We introduce the notion of \emph{positive length} of a strip $\textbf{s}$ to mean the number of tiles of that strip that are contained northwest of $P(X)$, and we denote it by $L_+(\textbf{s})$. (Note that all north strips and all west strips have the same total length of $r+k$ and $r+\ell$, respectively. Moreover, for a west strip, positive length 0 is equivalent to positive length $r+\ell$, and for a north strip, positive length 0 is equivalent to positive length $r+k$.)
\end{defn}

\begin{defn}
Let $\mathcal{H}(X)$ have tiling $\mathcal{T}$. The \emph{insertion of a north-strip} $s_N(x)$ at a point $x$ on $P(X)$ is defined as follows. Take the path $p_N(x)$ that begins and ends at $x$ by traveling north up the vertical or diagonal edges of $\mathcal{T}$. It is important that $p_N(x)$ is as far to the right as possible, meaning that, if at some point the path has a choice between taking a vertical edge or a diagonal edge, it always chooses the diagonal one. Now replace each vertical edge of $p_N(x)$ with a 20-tile and each diagonal edge of $p_N(x)$ with a 10-tile. The newly inserted tiles form the north-strip $s_N(x)$. The horizontal edge adjacent to $x$ becomes a new 0-edge in $P(X)$. See Figure \ref{markov_north}.
\end{defn}

\begin{figure}[!ht]
  \centerline{\includegraphics[height=1in]{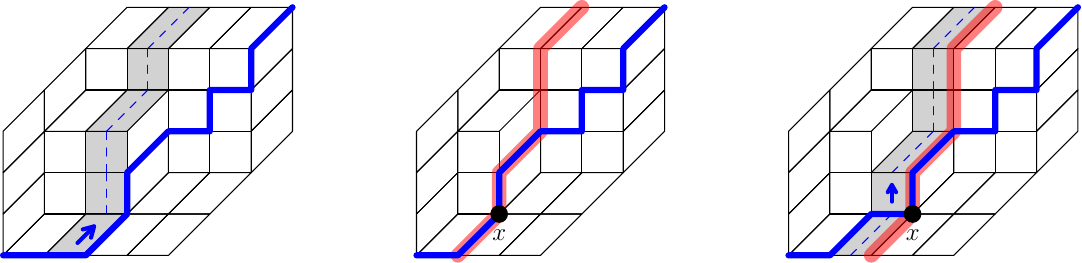}}
\centering
 \caption{The transition $\Omega^{\TRAT}_8$ from a TRAT $R$ of type $1202012\underline{\textbf{10}}0$ to a TRAT $R'$ of type $1202012\underline{\textbf{01}}0$ is shown. On the left is $R$ with a corner tile containing an up-arrow in north-strip $\textbf{n}$, and $L_+(\textbf{n})=6$. In the middle is the TRAT with the north-strip $\textbf{n}$ removed. Location $x$ is chosen since $L_+(s_N(x))=5$; $p_N(x)$ is marked by the pink path. On the right, $s_N(x)$ is inserted to build $R'=\Omega^{\TRAT}_8(R)$.}\label{markov_north}
 \end{figure}

\begin{defn}
Let $\mathcal{H}(X)$ have tiling $\mathcal{T}$. The \emph{insertion of a west-strip} $s_W(x)$ at a point $x$ on $P(X)$ is defined as follows. Take the path $p_W(x)$ that begins and ends at $x$ by traveling west along the horizontal or diagonal edges of $\mathcal{T}$. It is important that $p_W(x)$ is as far to the south as possible, meaning that, if at some point the path has a choice between taking a horizontal edge or a diagonal edge, it always chooses the diagonal one. 
Now replace each horizontal edge of $p_W(x)$ with a 20-tile and each diagonal edge of $p_W(x)$ with a 21-tile. The newly inserted tiles form the west-strip $s_W(x)$. The vertical edge adjacent to $x$ becomes a new 2-edge in $P(X)$. See Figure \ref{markov_west}.
\end{defn}

\begin{figure}[!ht]
  \centerline{\includegraphics[height=1in]{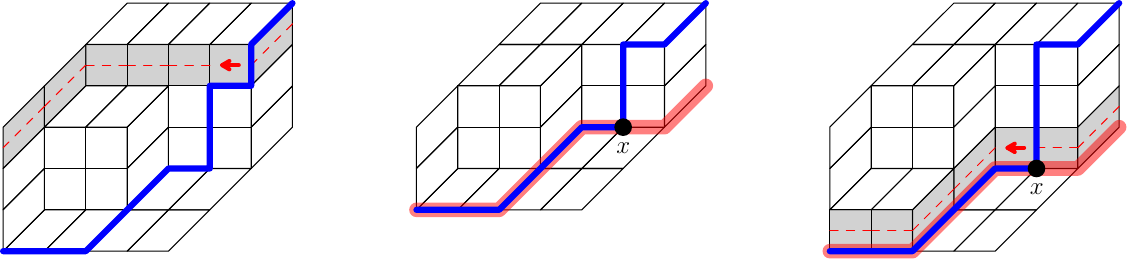}}
\centering
 \caption{The transition $\Omega^{\TRAT}_2$ from a TRAT $R$ of type $1\underline{\textbf{20}}2201100$ to a TRAT $R'$ of type $1\underline{\textbf{02}}2201100$ is shown. On the left is $R$ with a corner tile containing a left-arrow in west-strip $\textbf{w}$, and $L_+(\textbf{w})=6$. In the middle is the TRAT with the west-strip $\textbf{w}$ removed. Location $x$ is chosen since $L_+(S_W(x))=5$; $p_W(x)$ is marked by the pink path. On the right, $s_W(x)$ is inserted to build $R'=\Omega^{\TRAT}_2(R)$.}\label{markov_west}
 \end{figure}

\begin{defn} We define two types of transitions at a corner tile of TRAT $R$ containing an up-arrow or a left-arrow: the \emph{up-arrow transition} and the \emph{left-arrow transition} accordingly.
\begin{itemize}
\item \textit{Up-arrow transition}: let the up-arrow be contained in north-strip $\textbf{s}$. Let $x$ be the right-most location of $P(X)$ such that $L_+(s_N(x))=L_+(\textbf{s})-1 \mod (k+r)$. Remove $\textbf{s}$ from $R$, and insert the north-strip $s_N(x)$ at $x$, placing an up-arrow in its bottom-most box. See Figure \ref{markov_north}.
\item \textit{Left-arrow transition}: let the left-arrow be contained in west-strip $\textbf{s}$. Let $x$ be the bottom-most (i.e.~left-most) location of $P(X')$ such that $L_+(s_W(x))=L_+(\textbf{s})-1 \mod (\ell+r)$. Remove $\textbf{s}$ from $R$, and insert the west-strip $s_W(x)$ at $x$, placing a left-arrow in its right-most box. See Figure \ref{markov_west}.
\end{itemize}
\end{defn}

\begin{remark}
There is a subtlety arising from the choice of a tiling on $\mathcal{H}(X)$: there may be no 21-tile adjacent to a 21 corner, or there may be no 10-tile adjacent to a 10 corner. That can occur only when $P(X)$ has consecutive 2, 1, 0 edges in that order, in which case there will be a hexagonal configuration of three tiles adjacent to those three edges. In this case, we use the property of flip equivalence of Lemma \ref{flip_lemma}, to perform a flip on that configuration, placing the desired tile in the desired corner.
\end{remark}

\begin{defn}
Define $\Omega^{\TRAT}_i\ :\ \TRAT(k,r,\ell) \rightarrow \TRAT(k,r,\ell)$ to be the transition of $\Omega^{\TRAT}_{MC}$ on TRAT $R$ at the corner $(i,i+1)$ (where by convention we number the edges of $P(X)$ of a TRAT of type $X$ from right to left). 
\end{defn}

When $(i,i+1)$ is not a corner of $R$, $\Omega^{\TRAT}_i(R)$ is the identity map. Otherwise, we define $\Omega^{\TRAT}_i(R)$ as follows.

\textbf{(a.)\ \ \ $\textbf{(i, i+1)}$ is a 20 corner.}

There is necessarily a 20-tile containing either an up-arrow or a left-arrow adjacent to that corner. In the former case, the up-arrow transition is performed on $R$ to obtain $\Omega^{\TRAT}_i(R)$. In the latter case, the left-arrow transition is performed on $R$ to obtain $\Omega^{\TRAT}_i(R)$.

\textbf{(b.)\ \ \ $\textbf{(i, i+1)}$ is a 21 corner.}

If there is a 21-tile adjacent to that corner, it must necessarily contain a left-arrow. If there is no 21-tile adjacent to that corner, perform flips on $\mathcal{H}(X)$ until a 21-tile appears in the desired location. This tile must necessarily contain a left-arrow, and we perform the 21 transition on this new tiling. In both of these cases, the left-arrow transition is performed on $R$ to obtain $\Omega^{\TRAT}_i(R)$.

\textit{Special case when $i=n$ and $X=1Y2$.}

When $X=1Y2$ and we perform the transition $21\rightarrow12$ for the 1 and 2 edges at opposite ends of $P(X)$, we have a special case. The left-arrow in the bottom-most west-strip $\textbf{w}$ is contained in its rightmost 21-tile. We have $L_+(\textbf{w})=0$, and so $x$ is the bottom-most point of $P(12Y)$ such that $L_+(x)=\ell+r-1$. Then, as with the usual left-arrow transition, the west-strip $s_W(x)$ is inserted at $x$ with a left-arrow placed in its right-most box to obtain $\Omega^{\TRAT}_n(R)$. See Figure \ref{markov_west_exception} for an example. 

\begin{figure}[!ht]
  \centerline{\includegraphics[height=1in]{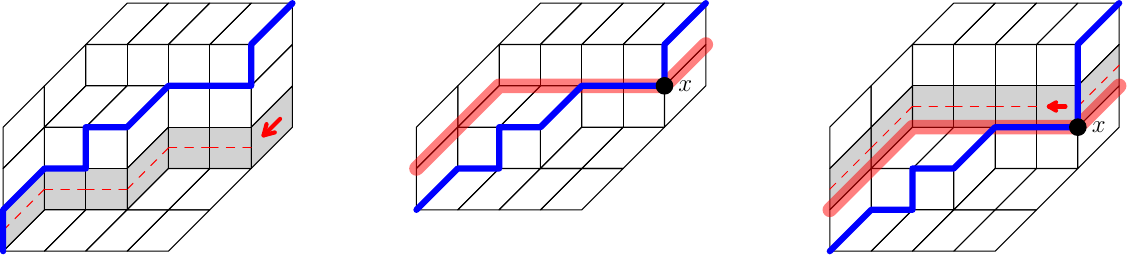}}
\centering
 \caption{The transition $\Omega^{\TRAT}_{10}(R)$ from a TRAT $R$ of type $\underline{\textbf{1}}20010201\underline{\textbf{2}}$ to a TRAT $R'$ of type $\underline{\textbf{12}}20010201$ is shown. On the left is $R$ with a corner tile containing a left-arrow in west-strip $\textbf{w}$, and $L_+(\textbf{w})=0$. In the middle is the TRAT with the west-strip $\textbf{w}$ removed. Location $x$ is chosen since $L_+(s_W(x))=6\equiv-1 \mod 7$; $p_W(x)$ is marked by the pink path. On the right, $s_W(x)$ is inserted to build $R'=\Omega^{\TRAT}_{10}(R)$.}\label{markov_west_exception}
 \end{figure}

\textbf{(c.)\ \ \ $\textbf{(i,i+1)}$ is a 10 corner.}

If there is a 10-tile adjacent to that corner, it must necessarily contain an up-arrow. If there is no 10-tile adjacent to that corner, perform flips on $\mathcal{H}(X)$ until such a tile appears in the desired location. This tile must necessarily contain an up-arrow, and we perform the 10 transition on this new tiling. In both of these cases, the up-arrow transition is performed on $R$ to obtain $\Omega^{\TRAT}_i(R)$.

\textit{Special case when $i=0$ and $X=10Y$.}

When $X=10Y$ and we perform the transition $10\rightarrow01$ for the 1 and 0 edges at the beginning of $P(X)$, we have a special case, since the 0-edge is then wrapped around to obtain a tableau of type $X'=1Y0$. The up-arrow in the right-most north-strip $\textbf{n}$ is contained in its bottom-most 10-tile. We have $L_+(\textbf{n})=1$, and so $x$ is the right-most point of $P(1Y0)$ such that $L_+(x)=r+k$. Then, as with the usual up-arrow transition, the north-strip $s_N(x)$ is inserted at $x$ with an up-arrow placed in its bottom-most box to obtain $\Omega^{\TRAT}_1(R)$. See Figure \ref{markov_north_exception} for an example. 

\begin{figure}[!ht]
  \centerline{\includegraphics[height=1in]{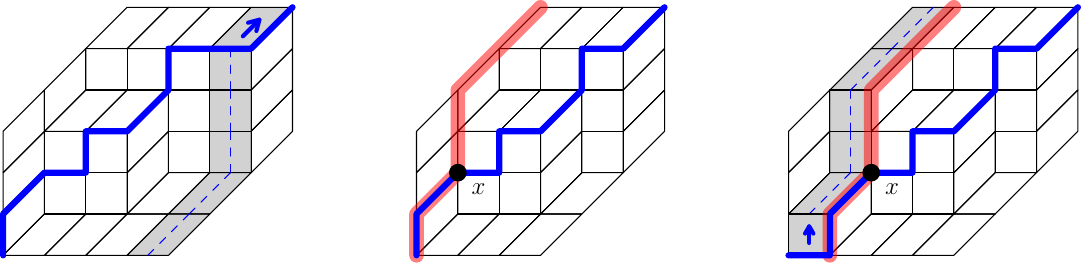}}
\centering
 \caption{The transition $\Omega^{\TRAT}_0(R)$ from a TRAT $R$ of type $\underline{\textbf{10}}02102012$ to a TRAT $R'$ of type $\underline{\textbf{1}}02102012\underline{\textbf{0}}$ is shown. On the left is $R$ with a corner tile containing an up-arrow in north-strip $\textbf{n}$, and $L_+(\textbf{n})=1$. In the middle is the TRAT with the north-strip $\textbf{n}$ removed. Location $x$ is chosen since $L_+(s_N(x))=6\equiv0 \mod 6$; $p_N(x)$ is marked by the pink path. On the right, $s_N(x)$ is inserted to build $R'=\Omega^{\TRAT}_0(R)$.}\label{markov_north_exception}
 \end{figure}

Our Markov chain is a projection onto the 2-TASEP on a ring if the following lemma holds.

\begin{lemma}\label{lem1}
If there is a transition $X \rightarrow Y$ on the TASEP chain, then for any TRAT $R_X$ of type $X$, there exists exactly one tableau $R_Y$ such that there is a transition $R_X \rightarrow R_Y$ on the TRAT chain and $R_Y$ has type $Y$.
\end{lemma}

\begin{proof}
For each $i$, the map $\Omega^{\TRAT}_i$ sends a tableau $R$ to a tableau $R'$, where the boundary of $R'$ is the boundary of $R$ with edges at locations $i$ and $i+1$ swapped. Thus if $X\rightarrow Y$ is a transition on the TASEP chain and $X$ differs from $Y$ at some adjacent pair of locations $(i, i+1)$, then $\Omega^{\TRAT}_i$ is the unique transition on the TRAT chain that sends a tableau of type $X$ to a tableau of type $Y$, so the lemma holds immediately by construction. 
\end{proof}

\begin{figure}[!ht]
  \centerline{\includegraphics[width=\linewidth]{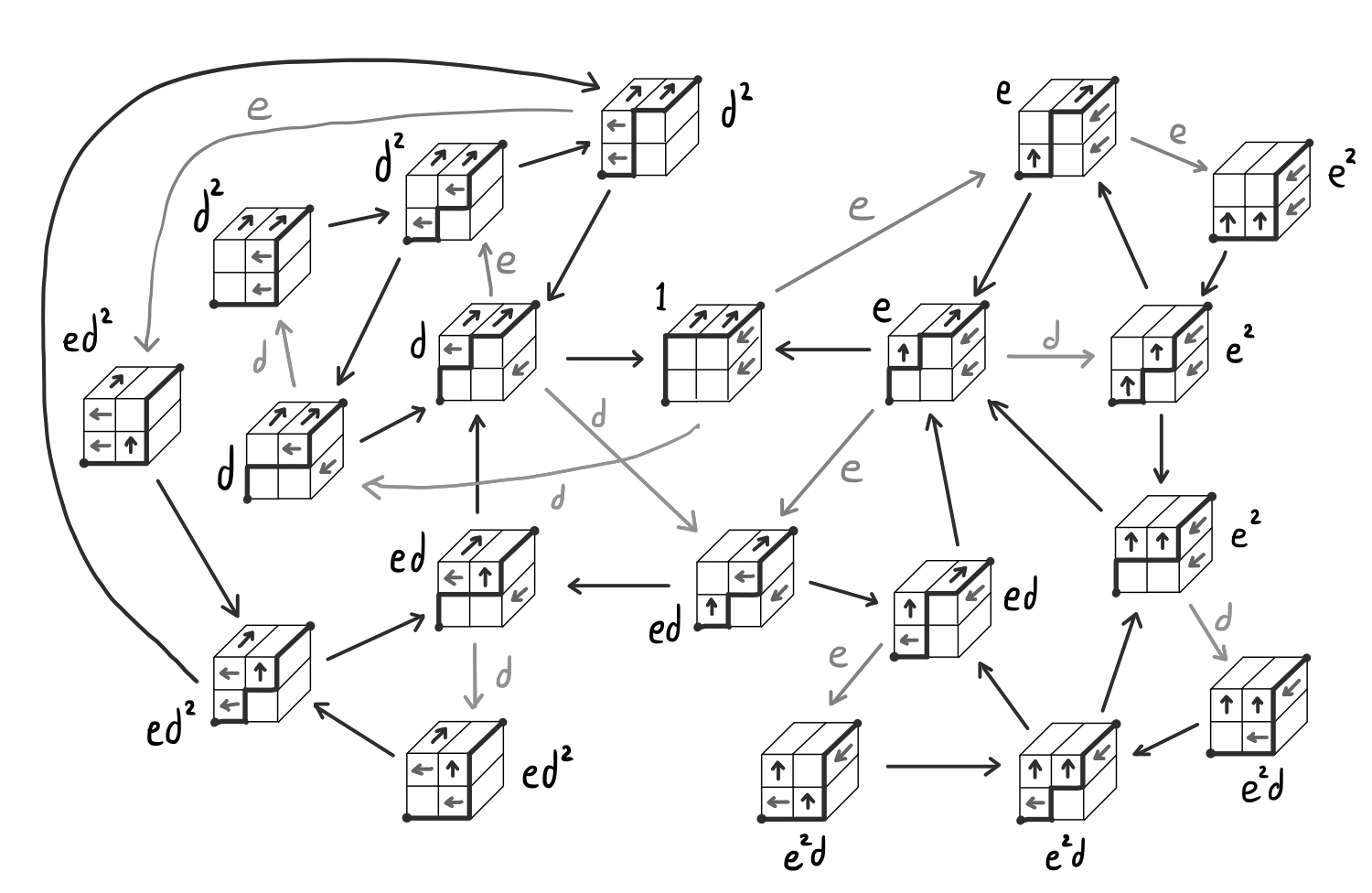}}
 \centering
  \caption{The Markov chain $\Omega^{\TRAT}_{MC}$ on states of size $(2,1,2)$. In this figure, the unmarked arrows represent transitions $20 \rightarrow 02$ which have rate $t=1$, while the arrows marked with $e$ represent transitions $10 \rightarrow 01$, and the arrows marked with $d$ represent transitions $21 \rightarrow 12$ (in this case the 1-edge and the 2-edge are at opposite ends of $P(X)$). The monomials labeling the tableaux give the weights of the tableaux according to Definition \ref{trat_weight_inhomog}.}
  \label{Markov_all}
\end{figure}

Denote the transition rate from $X$ to $X'$ by $\pr(X \rightarrow X')$.

\begin{defn}
$\wt$ satisfies \emph{detailed balance} on a Markov chain with states $\mathcal{S}$ if for all $X\in\mathcal{S}$,
\[\wt(X)\sum_{X'\in\mathcal{S}} \pr(X \rightarrow X') = \sum_{X''\in\mathcal{S}} \pr(X'')\pr(X'' \rightarrow X).\]
\end{defn}

If a Markov chain satisfies detailed balance, the weight of each state is proportional to its stationary distribution.

\begin{lemma}\label{lem2}
There is a uniform stationary distribution on the (homogeneous) TRAT chain.
\end{lemma}

\begin{proof}
When each transition has rate 1, detailed balance holds for a uniform distribution on the tableaux if and only if each tableau has an equal number of transitions going into and out of it.

By our definition of the TRAT chain, each corner of a tableau corresponds to precisely one transition coming out of it. By observing the image of the TRAT chain, we see that each corner also corresponds to precisely one transition going into the tableau (we omit the sufficiently straightforward definition of the reverse chain $(\Omega^{\TRAT}_{MC})^{-1}$ with transitions given by $(\Omega^{\TRAT}_i)^{-1}$ for $0\leq i \leq n$, which is simply the reverse of our definition of the forward chain - it will be further discussed in the next lemma). Thus the above claim holds true, which proves the lemma.
\end{proof}

As one would hope, the Markov chain on the weighted TRAT with transitions identical to the usual TRAT chain, projects to the inhomogeneous 2-TASEP on a ring. The following lemma shows that detailed balance holds in the inhomogeneous case.

\begin{lemma}\label{lem3}
Let $R$ be a TRAT with type $X$, and let $R'=\Omega^{\TRAT}_i$. There exists a tableau $R''$ such that $\wt(R)\pr(R \rightarrow R') = \wt(R'')\pr(R'' \rightarrow R)$.
\end{lemma}

\begin{figure}[!ht]
  \centerline{\includegraphics[width=.9\linewidth]{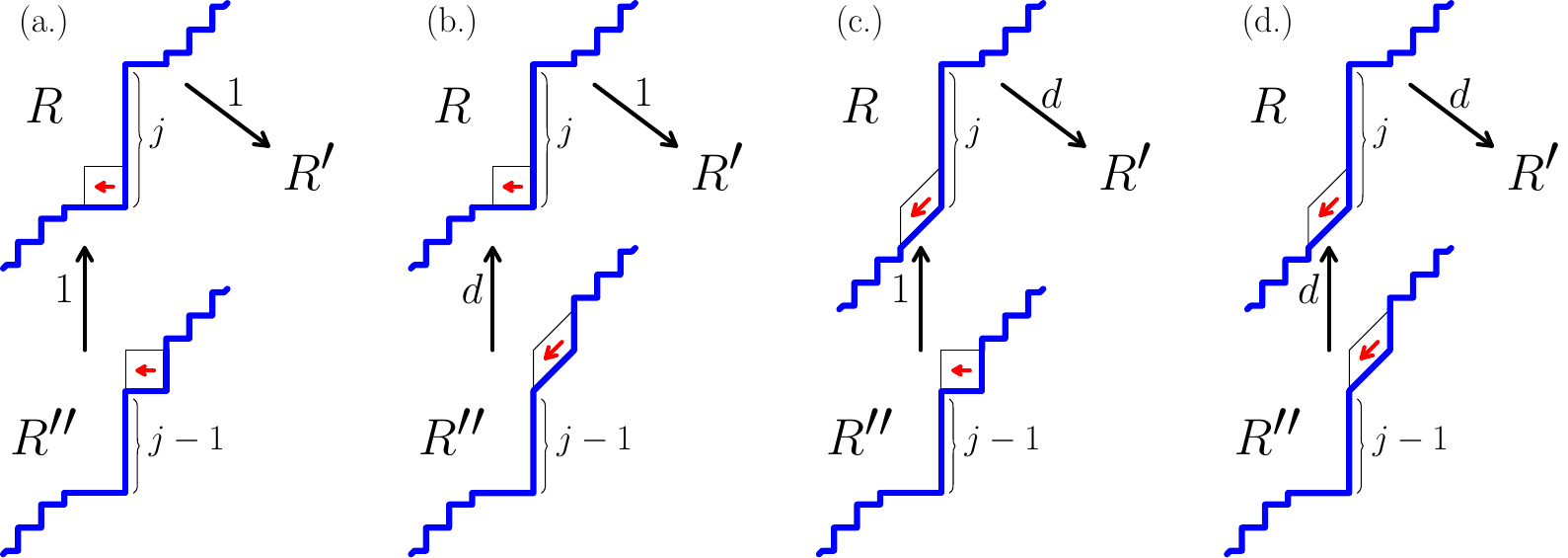}}
 \centering
  \caption{An illustration of Lemma \ref{lem3} is shown, with tableaux $R$, $R'$, and $R''$ such that $R\rightarrow R'$, $R''\rightarrow R$, and $\pr(R)\pr(R\rightarrow R')=\pr(R'')\pr(R''\rightarrow R)$, and the tiles at corners $\textbf{c}$ and $\textbf{i}$ highlighted.}
  \label{inhomog_trans}
\end{figure}

\begin{proof}
Let $\textbf{c}$ be the corner at edges $(i,i+1)$ of $R$. If $\textbf{c}$ contains a left-arrow, let $\textbf{u}$ be the closest inner corner to its right at edges $(g,g+1)$. If $\textbf{c}$ contains an up-arrow, let $\textbf{u}$ be the closest inner corner to its left at edges $(g,g+1)$. Let $R''=(\Omega^{\TRAT}_i)^{-1}(R)$ be the tableau obtained by performing a \emph{reverse TRAT chain transition} at $\textbf{c}$, which is equivalent to saying $R=\Omega^{\TRAT}_g(R")$. (This transition amounts to switching the edges of $\textbf{u}$ and converting it from an inner corner to an outer corner, and then moving the contents of $\textbf{c}$ in that new tile while correspondingly shifting the strips between $\textbf{c}$ and $\textbf{u}$.) 
We consider the following eight cases, the first four of which are shown in Figure \ref{inhomog_trans}.
\begin{enumerate}
\item[(a.)] $\textbf{c}$ is a 20-corner containing a left-arrow, and $\textbf{u}$ is a 02-corner. Then $\wt(R'')=\wt(R)$ since the left-arrow from $\textbf{c}$ is still in a 20-tile in $R''$, and $\pr(R'' \rightarrow R)= \pr(R \rightarrow R')=1$.
\item[(b.)] $\textbf{c}$ is a 20-corner containing a left-arrow, and $\textbf{u}$ is a 12-corner. Then $\wt(R'')=d^{-1}\wt(R)$ since $R''$ loses a 20-tile containing a left-arrow, and $\pr(R'' \rightarrow R)= d \pr(R \rightarrow R')=d$.
\item[(c.)] $\textbf{c}$ is a 21-corner containing a left-arrow, and $\textbf{u}$ is a 02-corner. Then $\wt(R'')=d\wt(R)$ since $R''$ gains a 20-tile containing a left-arrow and $\pr(R'' \rightarrow R)= d^{-1}\pr(R \rightarrow R')=1$.
\item[(d.)] $\textbf{c}$ is a 21-corner containing a left-arrow, and $\textbf{u}$ is a 12-corner. Then $\wt(R'')=\wt(R)$ since the left-arrow from $\textbf{c}$ is still in a 21-tile in $R''$, and $\pr(R'' \rightarrow R)= \pr(R \rightarrow R')=d$.
\item[(e.)] $\textbf{c}$ is a 20-corner containing an up-arrow, and $\textbf{u}$ is a 02-corner. Then $\wt(R'')=\wt(R)$ since the up-arrow from $\textbf{c}$ is still in a 20-tile in $R''$, and $\pr(R'' \rightarrow R)= \pr(R \rightarrow R')=1$.
\item[(f.)] $\textbf{c}$ is a 20-corner containing an up-arrow, and $\textbf{u}$ is a 01-corner. Then $\wt(R'')=e^{-1}\wt(R)$ since $R''$ loses a 20-tile containing an up-arrow, and $\pr(R'' \rightarrow R)= e \pr(R \rightarrow R')=e$.
\item[(g.)] $\textbf{c}$ is a 10-corner containing an up-arrow, and $\textbf{u}$ is a 02-corner. Then $\wt(R)=e\wt(R)$ since $R''$ gains a 20-tile containing an up-arrow, and $\pr(R'' \rightarrow R)= e^{-1}\pr(R \rightarrow R')=1$.
\item[(h.)] $\textbf{c}$ is a 10-corner containing an up-arrow, and $\textbf{u}$ is a 01-corner. Then $\wt(R'')=\wt(R)$ since the up-arrow from $\textbf{c}$ is still in a 10-tile in $R''$, and $\pr(R'' \rightarrow R)= \pr(R \rightarrow R')=e$.
\end{enumerate}
In all cases, $\wt(R)\pr(R \rightarrow R') = \wt(R'')\pr(R'' \rightarrow R)$, completing the proof.
\end{proof}

To conclude, Proposition \ref{mc_walk}, Lemma \ref{lem1}, and Lemma \ref{lem3} imply the following result.
\begin{cor}
The inhomogeneous TRAT chain, whose states have weights given by $\wt_e$, projects onto the inhomogeneous 2-TASEP on a ring.
\end{cor}

\begin{example}
Figure \ref{Markov_all}, shows an example of the Markov chain on all the states of $\TRAT(2,1,2)$ which projects to the inhomogeneous 2-TASEP with parameters $e$ and $d$. In this example, we obtain the following stationary probabilities:
\begin{align*}
\mathcal{Z}_{2,1,2}\Prob(10022)&=1&\qquad \mathcal{Z}_{2,1,2}\Prob(12020)&=d^2+ed^2+ed+e^2d+e^2\\
\mathcal{Z}_{2,1,2}\Prob(10202)&=d+e&\qquad \mathcal{Z}_{2,1,2}\Prob(12200)&=d^2+2ed^2+2e^2d+e^2\\
\mathcal{Z}_{2,1,2}\Prob(10220)&=d^2+ed+e&\qquad \mathcal{Z}_{2,1,2}\Prob(12002)&=d+ed+e^2\\
\end{align*}
with $\mathcal{Z}_{2,1,2}=1+2d+2e+3ed+3d^2+3e^2+3ed^2+3e^2d$.
\end{example}

\subsection{Markov Chain on multiline queues that projects to the inhomogeneous 2-TASEP on a ring}\label{sec_MLQ_markov}

From the Markov chain on the TRAT and following the bijection of Section \ref{sec_bij}, we construct a minimal Markov chain on the weighted MLQs, which we call $\Omega^{\MLQ}_{MC}$, that is different from both Markov chains in \cite{FM07} and in \cite{AL14}, that projects to the inhomogeneous 2-TASEP on a ring. The Markov chain is minimal in the sense that every nontrivial transition in $\Omega^{\MLQ}_{MC}$ corresponds to a nontrivial transition in the TASEP.

Let $M$ be an MLQ. 
We denote by $M(i)$ the TASEP particle corresponding to location $i$ in $M$. Recall that if the bottom row contains a vacancy at location $i$, then $M(i)=2$; if the bottom row contains a 0-ball  (i.e.~one that is hit by a dropping top row ball), $M(i)=0$; and if the bottom row contains a 1-ball (i.e.~that is not hit by a dropping top row ball), $M(i)=1$.

\begin{defn}
We call a ball in the bottom row \emph{occupied} if there is a ball directly above it. Otherwise if there is a vacancy above it, we call it \emph{vacant}.
\end{defn}

Note that 1-balls are necessarily vacant. Moreover, no path from a top row ball to the 0-ball it occupies can pass through a 1-ball. 

\begin{defn}\label{def_MLQ_mc}
The transition of $\Omega^{\MLQ}_{MC}$ on $M$ at location $i$, denoted by $\Omega^{\MLQ}_i(M)$, is given by the following rules.
\begin{itemize}
 \item \emph{Occupied jump:} if a transition occurs at an occupied ball at location $i$, let $j<i-1$ be the nearest index left of $i-1$ such that $M(j)\neq 0$, that is, $j = \max\{j<i-1: M(j)\neq0\}$. A ball is inserted in the top row at location $j+1$, shifting all top row contents at locations $j+1,\ldots,i-1$ one spot to the right.
 \item \emph{Vacant jump:} if a transition occurs at a vacant ball at location $i$, let $j>i$ be the nearest index right of $i$ such that $M(j)\neq 2$, that is, $j = \min\{j>i: M(j)\neq2\}$. A vacancy is inserted in the top row at location $j$, shifting all top row contents at locations $i+1,\ldots,j-1$ one spot to the left.
\end{itemize}
In both cases, the bottom row contents of locations $i-1$ and $i$ are swapped.
\end{defn}

Figure \ref{MLQ_transitions} shows examples of each of the occupied and vacant jumps.

\begin{figure}[!h]
  \centerline{\includegraphics[width=\linewidth]{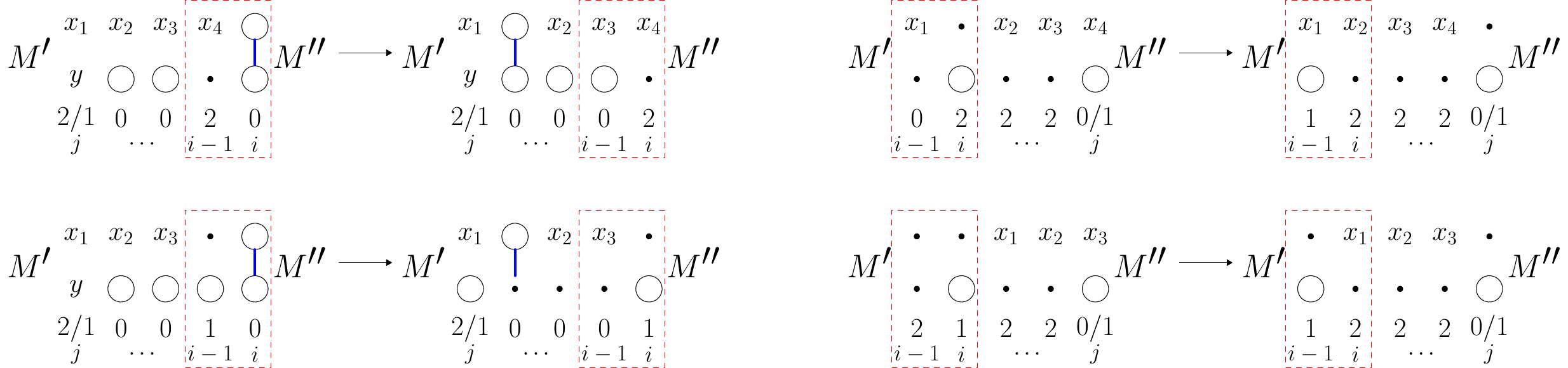}}
\centering
 \caption{On the left are the occupied jumps, showing a $20 \rightarrow 02$ transition on top and a $10 \rightarrow 01$ transition on the bottom. On the right are the vacant jumps, showing a $20 \rightarrow 02$ transition on top and a $21 \rightarrow 12$ transition on the bottom. For clarity, in this figure dots represent vacancies (in the top and bottom row), and the $x_i$'s represent arbitrary entries. The variable $y$ represents either a vacancy or a 1-ball.}\label{MLQ_transitions}
 \end{figure}

\begin{remark}
Though $\Omega^{\MLQ}_{MC}$ has some similarities with the minimal Markov chain described in Section 5 of \cite{AL14}, out transitions are different. In particular, our Markov chain is equivalent to the TRAT Markov chain through the MLQ-TRAT bijection, which is addressed in the following lemma. 
\end{remark}

\begin{lemma}
Let $M$ be an MLQ. The MLQ-TRAT bijection gives the following correspondences.
\begin{itemize}
\item An occupied jump on $M$ at location $i$ corresponds to an up-arrow transition in $\Omega^{\TRAT}_{MC}$ at edges $(i-1,i)$, which is a 20 transition (resp.~10 transition) at locations $(i-1,i)$ on the ASEP chain when $M(i-1)$ is a vacancy (resp.~1-ball). 
\item A vacant jump on $M$ at location $i$ corresponds to a left-arrow transition in $\Omega^{\TRAT}_{MC}$ at edges $(i-1,i)$, which is a 20 transition (resp.~21 transition) at locations $(i-1,i)$ on the ASEP chain when $M(i)$ is a 0-ball (resp.~1-ball). 
\end{itemize}
This implies that for all $i$, we have $\trat(\Omega^{\MLQ}_i(M))=\Omega^{\TRAT}_{i-1}(\trat(M))$.
\end{lemma}

\begin{proof}
Let $M$ be an MLQ and $T=\trat(M)$, and let $M'=\Omega^{\MLQ}_i(M)$ and $T'=\trat(M')$. 

If a ball at location $i$ is occupied, $M(i)=0$, and furthermore by the definition of our bijection there are no left-arrows in column $i$ of $T$, and hence an up-arrow is contained in the bottom-most possible tile. Thus if the edges $(i-1,i)$ are at a corner of $T$, that corner tile contains an up-arrow. Now, since $j<i-1$ is the largest index such that $M(j)\neq 0$, $M(j+1)=0$. Thus our definition of the occupied jump is equivalent to removing the column $i$ from $M$, and inserting it to the right of column $j$. On the other hand, in $T'$ that means removing column $i$ with the up-arrow in its bottom-most tile and re-inserting it to the left of edge $j$, where $j$ is the closest non-horizontal edge to the right of $i-1$. The new column has an up-arrow in its bottom-most box since $M'$ has an occupied ball at location $j$. The rest of $M'$ is left unchanged from $M$, and hence the rest of $T'$ is left unchanged from $T$. This is precisely the definition of the TRAT transition at corner $(i-1, i)$ with an up-arrow in that corner, so $T'=\Omega^{\TRAT}_{i-1}(T)$.

If a ball at location $i$ is vacant and $M(i)=0$ (resp.~$M(i)=1$), there is a left-arrow in the corner 20-tile (resp.~21-tile) at edges $(i-1,i)$ of $T$. (Note that If $M(i)=1$, we assume the tiling of $T$ has a 21-tile at the $(i-1,i)$ corner. If the tiling does not have such a tile, we perform filling-preserving flips until it does.) Since $j>i$ is the smallest index such that $M(j)\neq 2$, $M(j-1)=2$. Recall that in a vacant jump, every top row entry from column $j$ to $i-1$ is shifted one location to the right, and a vacancy is placed in the top row of column $j$. In particular, if $M(j)=0$, the hitting weight of the ball at location $j$ increases by 1, while keeping all others unchanged. On the other hand, in $T'$ that means removing row $i$ with the left-arrow in its right-most tile and re-inserting this row to the right of edge $j$, where $j$ is the closest non-vertical edge to the left of $i-1$. (We assume the row is inserted such that the tiling of $T'$ is standard.) Since the hitting weight of the ball at location $j$ increased by 1 while keeping all others unchanged, $T'$ has an extra left-arrow in column $j$, which corresponds precisely to inserting a row with a left-arrow in its right-most box to the right of edge $j$. The latter is precisely the definition of the TRAT transition at corner $(i-1, i)$ with a left-arrow in that corner, so $T'=\Omega^{\TRAT}_{i-1}(T)$, thus completing the proof.
\end{proof}

\subsection{Markov chain on acyclic multiline queues that projects to the inhomogeneous 2-TASEP with open boundaries}

There is a Markov chain on the acyclic MLQs, which has the same bulk transitions as $\Omega^{\MLQ}_{MC}$, that projects to the 2-TASEP with open boundaries. This Markov chain is obtained directly by pushing the Markov chain $\zeta^{\RAT}_{MC}$ on RAT from \cite{Man15a} through the $\RAT \rightarrow \AMLQ$ bijection. We call this Markov chain $\Omega^{\AMLQ}_{MC}$, which is defined by transitions $\Omega^{\AMLQ}_i$ for $1 \leq i \leq n+1$.

\begin{defn}
Let $A\in \AMLQ(n,r)$. For $2\leq i\leq n$, we define $\Omega^{\AMLQ}_i(A)=\Omega^{\MLQ}_i(A)$. For $i=1$ and $i=n+1$, we define the left and right boundary transitions $\Omega^{\AMLQ}_1$ and $\Omega^{\AMLQ}_{n+1}$ as follows, with examples shown in Figure \ref{AMLQ_fig}.
\begin{itemize}
\item \textbf{Left boundary transition $\Omega^{\AMLQ}_1$:} If $A$ has a 0-ball at its left boundary, let $j>1$ be the nearest index that does not contain a bottom row vacancy. The leftmost 0-ball is necessarily occupied by a ball above it. Replace the leftmost bottom row 0-ball by a vacancy, remove the leftmost top row ball, and shift all top row contents left of location $j+1$ one location to the left, and insert a vacancy in the top row of location $j$.
\item \textbf{Right boundary transition $\Omega^{\AMLQ}_{n+1}$:} If $A$ has a vacancy at its right boundary, let $j<n$ be the nearest index that does not contain a 0-ball. There is necessarily a top row vacancy above the rightmost bottom row vacancy. Remove the rightmost column and insert a column consisting of a 0-ball occupied by a ball above it at location $j+1$.
\end{itemize}
\end{defn}

\begin{figure}[!ht]
 \centerline{\includegraphics[width=\linewidth]{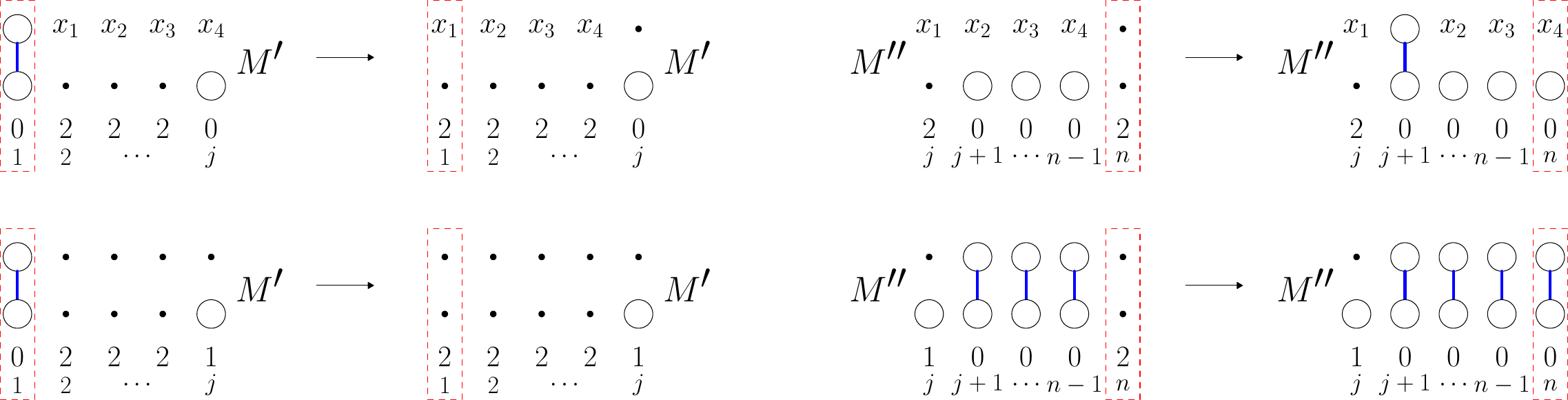}}
\centering
 \caption{On top, we see a left boundary transition on the left and a right boundary transition on the right. On the bottom we see special cases of each when $M(j)=1$.}\label{AMLQ_fig}
 \end{figure}

To show $\Omega^{\AMLQ}_{MC}$ indeed projects onto the inhomogeneous 2-TASEP with open boundaries, we use the fact that the AMLQ-RAT bjection is weight-preserving by the proof of Theorem \ref{amlq_main}, and refer back to the Markov chain on RAT from \cite{Man15a} combined with our proof of Lemma \ref{lem3}. 

\begin{thm}[\cite{Man15a}]
There is a Markov chain $\zeta^{\RAT}_{MC}$ on $\RAT(n,r)$, where each $R\in \RAT(n,r)$ has weight $\wt(R)$, that projects to the 2-TASEP with open boundaries of size $(n,r)$.
\end{thm}

We briefly describe the transitions of $\zeta^{\RAT}_{MC}$, and refer to \cite{Man15a} for proofs and technical details. 

\begin{defn}\label{RAT_markov_def}
The transitions of $\zeta^{\RAT}_{MC}$ are maps
\[\zeta^{\RAT}_i: \RAT(n,r)\rightarrow \RAT(n,r)
\]
for $0\leq i\leq n$, and are defined as follows.

For $1\leq i<n$, let $\zeta^{\RAT}_i$ be a transition occurring at edges $(i,i+1)$; in the 2-TASEP word $X=X_1\ldots X_n$, this corresponds to the transition $X_iX_{i+1}\rightarrow X_{i+1}X_i$.  The boundary transition at the first edge of the RAT is $\zeta^{\RAT}_0$, which is the transition $0X' \rightarrow 2X'$ in the 2-TASEP chain. The boundary transition at the last edge of the RAT is $\zeta^{\RAT}_n$, which is the transition $X'2 \rightarrow X'0$ in the 2-TASEP chain.

For $1\leq i<n$, if $X_i>X_{i+1}$, assume the tiling $\mathcal{T}_X$ has an $X_iX_{i+1}$-tile adjacent to the corresponding corner. We have two possible cases for the contents of that tile. 
\begin{itemize}
\item If the tile contains an up-arrow and is in a north-strip of length $\geq 2$, $\zeta^{\RAT}_i(R)$ is a RAT obtained by removing the north-strip beginning at the corner, shortening it by tile, and re-inserting it in the rightmost possible location with an up-arrow still in its bottom-most tile. If the north-strip had length 1 to start, it is reinserted as a single horizontal edge at the rightmost point of $R$.
\item If the tile contains a left-arrow and is in a west-strip of length $\geq 2$, $\zeta^{\RAT}_i(R)$ is a RAT obtained by removing the west-strip beginning at the corner, shortening it by tile, and re-inserting it in the bottom-most possible location with a left-arrow still in its right-most tile. If the strip had length 1, it is reinserted as a single vertical edge at the leftmost point of $R$.
\end{itemize}

For $i=0$, the rightmost boundary edge of $R$ must be horizontal. To obtain $\zeta^{\RAT}_0(R)$, this edge is removed, and instead a west-strip of greatest possible length is inserted, while preserving the semi-perimiter of the RAT. The strip is inserted in the lowest possible location and a left-arrow is placed in its rightmost tile. 

For $i=n$, the leftmost boundary edge of $R$ must be vertical. To obtain $\zeta^{\RAT}_n(R)$, this edge is removed, and instead a north-strip of greatest possible length is inserted, while preserving the semi-perimiter of the RAT. The strip is inserted in the rightmost possible location and an up-arrow is placed in its bottom-most tile. 

In all other cases, $\zeta^{\RAT}_i$ is trivial. Figure \ref{RAT_markov} shows examples of each of these transitions. Observe that for $1\leq i <n$, the transitions $\zeta^{\RAT}_i$ and $\Omega^{\TRAT}_i$ are essentially identical. 
\end{defn}

\begin{figure}[!ht]
 \centerline{\includegraphics[width=.8\linewidth]{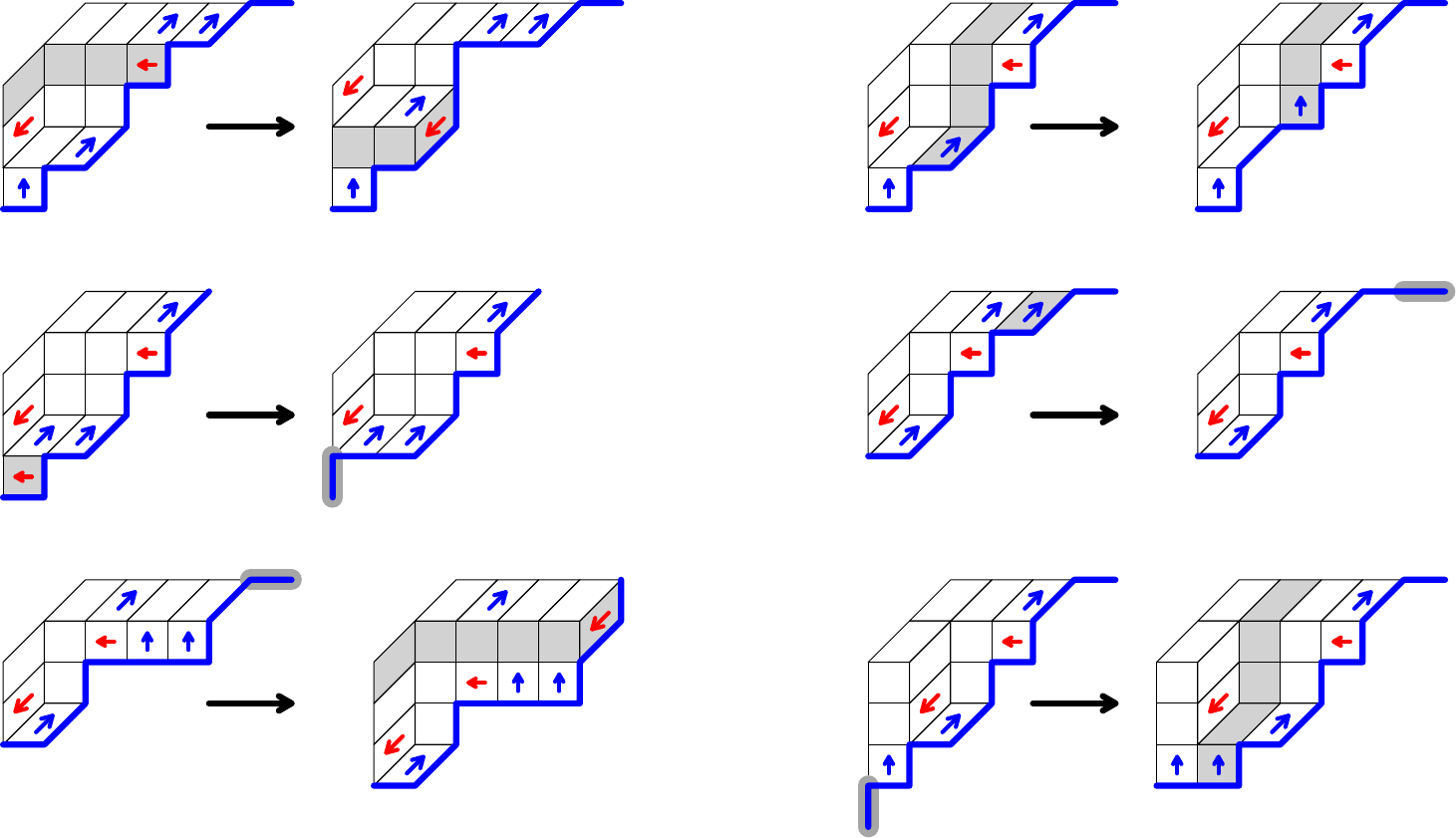}}
\centering
 \caption{Examples of some of the possible transitions on the RAT Markov chain $\zeta^{\RAT}_{MC}$, with left-arrow transitions shown on the left, up-arrow transitions on the right, and boundary transitions on the bottom. The highlighted strips are those which are removed and subsequently reinserted}\label{RAT_markov}
 \end{figure}

We will show the following.

\begin{prop}\label{rat_inhomog_markov}
The Markov chain $\zeta^{\RAT}_{MC}$ on $\RAT(n,r)$, where each $R\in \RAT(n,r)$ has weight $\wt_e(R)$, and whose transitions are given by parameters $\alpha$, $\beta$, $t=1$, $d$, and $e$ projects to the inhomogeneous 2-TASEP with open boundaries of size $(n,r)$.
\end{prop}

\begin{proof}
As in the proof of Lemma \ref{lem3} for the analogous result for the TRAT, the strategy of our proof is to show that detailed balance is preserved when each $R\in\RAT(n,r)$ has $\pr(R)=\wt_e(R)$. Namely, let $R'=\zeta^{\RAT}_i(R)$ for some $i$ such that $R\neq R'$. Then there exists some $R''\in\RAT(n,r)$ such that $\wt_e(R)\pr(R\rightarrow R')=\wt_e(R'')\pr(R''\rightarrow R)$. As in the proof of \ref{lem3}, we use the fact that the reverse Markov transitions of $\zeta^{\RAT}_{MC}$ are well-defined and set $R''=(\zeta^{\RAT}_i)^{-1}(R)$. 

There are sixteen possible cases for such triples $R$, $R'$, and $R''$, which we describe below.

\begin{figure}[!ht]

 \centerline{\includegraphics[width=.9\linewidth]{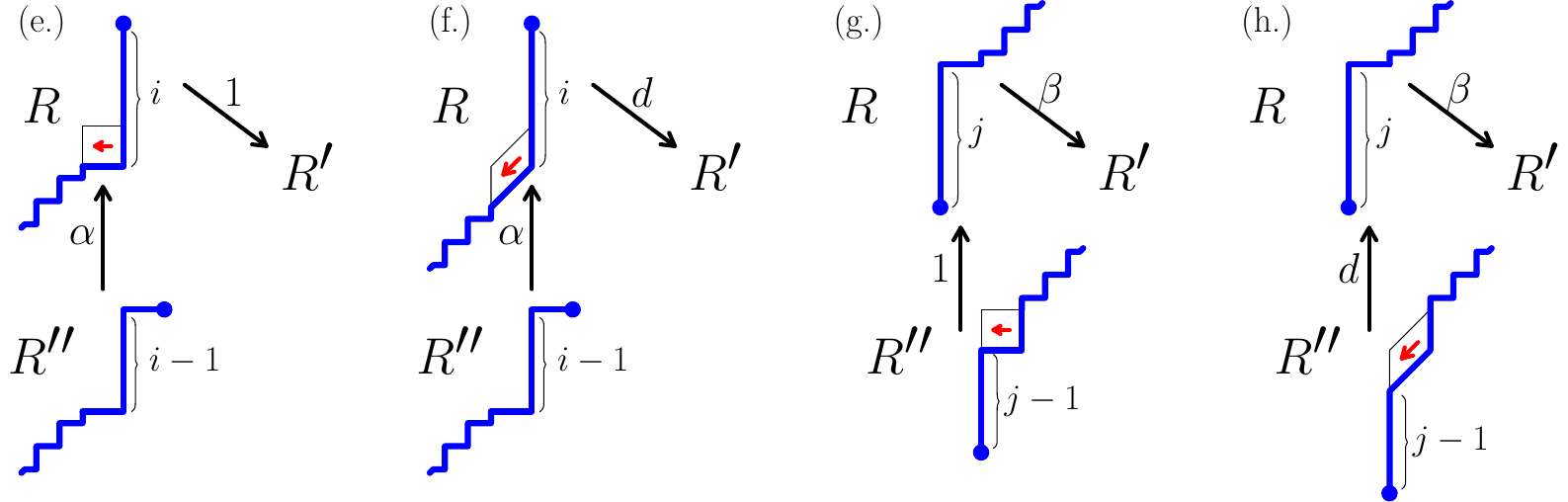}}
\centering
 \caption{This figure shows the triple $R$, $R'$, and $R''$ that satisfies $\wt(R)\pr(R\rightarrow R')=\wt(R'')\pr(R''\rightarrow R)$ for the cases (e)-(h) of Lemma \ref{rat_inhomog_markov}. The highlighted tiles and edges represent the location at which the Markov transition occurs.}\label{detailed1}
 \end{figure}

\textbf{Left-arrow transition at corner $\textbf{(i, i+1)}$:} all the cases are illustrated in Figures \ref{inhomog_trans} and \ref{detailed1}.
\begin{enumerate}
\item[(a.)] $X=Y02^j0Z$ with $|Y|=i-j-1$. Then $X''=Y202^{j-1}0Z$ and in $R$ the left-arrow moves from the rightmost tile of strip $i$ to the rightmost tile of strip $i-j$ on $\Gamma(X'')$ to form $R''$. Then all statistics of $R$ and $R''$ are equal so $\wt(R)=\wt(R'')$. $\pr(R\rightarrow R')=\pr(R''\rightarrow R)=1$.
\item[(b.)] $X=Y12^j0Z$, $X''=Y212^{j-1}0Z$ with $|Y|=i-j-1$; then $\Left(R'')=\Left(R)-1$ with all other statistics equal. Then $\wt(R'')=d^{-1}\wt(R)$ and $\pr(R''\rightarrow R)=d\pr(R\rightarrow R')=d$.
\item[(c.)] $X=Y02^j1Z$, $X''=Y202^{j-1}1Z$ with $|Y|=i-j-1$; then $\Left(R'')=\Left(R)+1$ with all other statistics equal. Then $\wt(R'')=d\wt(R)$ and $\pr(R''\rightarrow R)=d^{-1}\pr(R\rightarrow R')=1$.
\item[(d.)] $X=Y12^j1Z$. Then $X''=Y212^{j-1}1Z$ with $|Y|=i-j-1$; then $\Left(R'')=\Left(R)$ with all other statistics equal. Then $\wt(R'')=\wt(R)$ and $\pr(R''\rightarrow R)=\pr(R\rightarrow R')=d$.
\item[(e.)] $X=2^i0Z$. Then $X''=02^{i-1}0Z$ and in $R$ the left-arrow is removed from the rightmost tile of strip $i$ and replaced by a 0-edge at the right of $\Gamma(X'')$ to form $R''$. Then $\Left(R'')=\Left(R)-1$, $\ufree(R'')=\ufree(R)+1$, $R''$ has size $(k-1,r,\ell+1)$, and all other statistics of $R$ and $R''$ are equal. Then $\wt(R'')=\alpha^{-1}\wt(R)$ and $\pr(R\rightarrow R')=\alpha\pr(R''\rightarrow R)=\alpha$.
\item[(f.)] $X=2^i1Z$. Then $X''=02^{i-1}1Z$ and in $R$ the left-arrow is removed from the rightmost tile of strip $i$ and replaced by a 0-edge at the right of $\Gamma(X'')$ to form $R''$. Then $\ufree(R'')=\ufree(R)+1$, $R''$ has size $(k-1,r,\ell+1)$ with all other statistics equal. Then $\wt(R'')=d\alpha^{-1}\wt(R)$ and $\pr(R\rightarrow R')=\alpha d^{-1}\pr(R''\rightarrow R)=\alpha$.
\item[(g.)] $i=n$ and $X=Z02^{j}$. Then $X''=Z202^{j-1}$ and in $R$ the bottom-most 2-edge of $R$ is removed and is replaced by a left-arrow in the rightmost tile of strip $n-j$ of $\Gamma(X'')$ to form $R''$. Then $\Left(R'')=\Left(R)+1$ and $\lfree(R'')=\lfree(R)-1$,  with all other statistics equal. Then $\wt(R'')=\beta\wt(R)$ and $\pr(R\rightarrow R')=\beta^{-1}\pr(R''\rightarrow R)=1$.
\item[(h.)] $i=n$ and $X=Z12^{j}$. Then $X''=Z212^{j-1}$ and in $R$ the bottom-most 2-edge of $R$ is removed and is replaced by a left-arrow in the rightmost tile of strip $n-j$ of $\Gamma(X'')$ to form $R''$. Then $\lfree(R'')=\lfree(R)-1$, with all other statistics equal. Then $\wt(R'')=\beta d^{-1}\wt(R)$ and $\pr(R\rightarrow R')=d\beta^{-1}\pr(R''\rightarrow R)=d$.
\end{enumerate}

\textbf{Up-arrow transition at corner $\textbf{(i, i+1)}$}

By symmetry, we get this case for free: if we take (a)-(h) for the left-arrow transition, read $X$ and $X''$ from right to left, swap 2 with 0, swap up-arrows with left arrows, swap $\alpha$ with $\beta$, swap $d$ with $e$, swap $\Left()$ with $\Up()$, and swap $\lfree()$ with $\ufree()$, we obtain precisely the eight cases for the up-arrow transition.

Thus we found $R''$ given $R\rightarrow R'$ such that $\wt(R)\pr(R\rightarrow R')=\wt(R'')\pr(R''\rightarrow R)$ holds.
\end{proof}

\begin{lemma}\label{lem4}
Let $A\in\AMLQ$. Then $\rat(\Omega^{\AMLQ}_i(A))=\zeta^{\RAT}_{i-1}(\rat(A))$.
\end{lemma}

\begin{proof}
The transitions $\Omega^{\TRAT}_i$ and $\zeta^{\RAT}_i$ are identical when $1 \leq i<n$ and the strip containing the arrow in tile $(i,i+1)$ strip has length $\geq 2$. Following our bijections, the corresponding transitions $\Omega^{\MLQ}_{i+1}$ and $\Omega^{\AMLQ}_{i+1}$ are identical, as well. It remains to check the following cases: 
\begin{enumerate}
\item[(i.)] The strip containing the arrow in tile $(i, i+1)$ has length 1,
\item[(ii.)] $i=0$, and 
\item[(iii.)] $i=n$.
\end{enumerate}

Let $R=\rat(A)$ with type $X$. We show the following:
\begin{enumerate}
\item For some $j\geq 0$, if $X=0^j20Y'$ or $X=0^j10Y'$, the transition $\Omega^{\AMLQ}_{j+2}$ on $A$ corresponds to the up-arrow transition $\zeta^{\RAT}_{j+2}$ on $R$. 
\item For some $j\geq 0$, if $X=Y'202^j$, or $X=Y'212^j$, the transition $\Omega^{\AMLQ}_{n-j}$ on $A$ corresponds to the left-arrow transition $\zeta^{\RAT}_{n-j}$ on $R$.
\item $\Omega^{\AMLQ}_1$ on $A$ corresponds to the right boundary transition $\zeta^{\RAT}_0$ on $R$, from Definition \ref{RAT_markov_def}. 
\item $\Omega^{\AMLQ}_{n+1}$ on $A$ corresponds to the left boundary transition $\zeta^{\RAT}_n$ on $R$, from Definition \ref{RAT_markov_def}.
\end{enumerate}

For (1.), the AMLQ $A$ must have occupied 0-balls in its first $j$ columns (from the left), with an occupied 0-ball at location $j+2$. Then $\Omega^{\AMLQ}_{j+2}(A)$ is an AMLQ with occupied 0-balls in its first $j+1$ columns, followed by a vacancy or a 1-ball, with the rest of the AMLQ identical to $A$. This is precisely equivalent to removing the tile at the $(n-j-1, n-j)$ corner from $R$, which is indeed equal to $\zeta^{\RAT}_{n-j-1}$.

For (2.), the AMLQ $A$ must have vacancies in the top and bottom row in its rightmost $j$ columns, with a vacant 0-ball at location $j+2$. Then $\Omega^{\AMLQ}_{n-j}(A)$ is an AMLQ with vacancies in its rightmost $j+1$ columns, followed (on the left) by a 0-ball or a 1-ball, with the rest of the AMLQ identical to $A$. This is precisely equivalent to removing the tile at the $(j+1, j+2)$ corner from $R$, which is indeed equal to $\zeta^{\RAT}_{j+1}$.

For (3.), we let $u=\min\{u>1:\ A(u)\neq 2\}$. Inserting a vacancy above the ball at location $u$ in $\Omega^{\AMLQ}_1(A)$ results in a left-arrow being added to the north-strip $u$ in $\zeta^{\RAT}_0(R)$. Since $(u-1, u)$ is a corner, and since we have replaced the first 0-ball with a vacancy, this is equivalent to adding a west-strip with a left-arrow in its rightmost tile at the $(u-1, u)$ corner, which is precisely the definition of $\zeta^{\RAT}_0$. 

For (4.), we let $u=\max\{u<n:\ A(u)\neq 0\}$. A column of an MLQ consisting of an occupied 0-ball corresponds to a north-strip with an up-arrow in the bottom-most free location. Since $(u, u+1)$ is a corner of $\zeta^{\RAT})_n(R)$, a column in $\Omega^{\AMLQ}_{n+1}(A)$ consisting of an occupied 0-ball at location $u+1$ corresponds to a north-strip with an up-arrow in its bottom-most tile adjacent to that corner. This is precisely the definition of $\zeta^{\RAT}_n$.
\end{proof}

\section{Concluding remarks}

The tableaux method in this paper has a few advantages. First, it allows us to solve a more general, symmetric version of the inhomogeneous TASEP on a ring with three parameters for the hopping rates. Second, it establishes a connection between the well-studied multiline queue method of Ferrari and Martin solving the multispecies TASEP on a ring, and the alternative tableaux method originally introduced by Corteel and Williams. The multiline queues give combinatorics for the multispecies TASEP on a ring for any number of species but only for $q=0$. On the other hand, thus far the tableaux method has only been useful the 2-ASEP with open boundaries, albeit for general $q$. Our bijection makes us hopeful to find tableaux combinatorics for the two-species ASEP on a ring with general $q$. Furthermore, it would be interesting to put weights on the multiline queues or the acyclic MLQs to incorporate the $q$ parameter.

\textbf{Acknowledgement.} I am grateful to Sylvie Corteel and Lauren Williams for the inspiration, as well as many useful conversations. I would also like to thank Arvind Ayyer and Svante Linusson for discussion and useful input. The author was partially supported by NSF grant DMS-1704874 and the UC Presidential Postdoctoral Fellowship Program at UCLA during the completion of this work.

\end{document}